\documentclass[11pt]{article}

\usepackage[margin=2cm]{geometry}
\usepackage{mathtools}
\usepackage{latexsym}
\usepackage{amssymb,amsmath,amsthm}
\usepackage[svgnames,x11names,table]{xcolor}
\usepackage{graphicx}
\usepackage{calc}
\usepackage{hyphenat}
\usepackage{verbatim}

\usepackage{hyperref}

\graphicspath{{fig/}}

\newcommand*{\NN}{\ensuremath{\mathbb{N}}}
\newcommand*{\QQ}{\ensuremath{\mathbb{Q}}}
\newcommand*{\ZZ}{\ensuremath{\mathbb{Z}}}
\newcommand*{\RR}{\ensuremath{\mathbb{R}}}
\newcommand*{\CC}{\ensuremath{\mathbb{C}}}
\newcommand*{\FF}{\ensuremath{\mathbb{F}}}

\newcommand*{\PP}{\ensuremath{\mathcal{P}}}

\newcommand*{\LL}{\ensuremath{\mathcal{L}}}

\newcommand*{\from}{\colon}
\newcommand*{\ceq}{\coloneqq}

\newcommand*{\eps}{\varepsilon}
\newcommand*{\set}[1]{\{{#1}\}}
\newcommand*{\dset}[1]{\{{#1}\}}
\newcommand*{\tup}[1]{({#1})}
\newcommand*{\dtup}[1]{({#1})}
\newcommand*{\norm}[1]{ \|  #1 \|}
\newcommand*{\indicator}{1}

\newcommand*{\subim}{\le}
\newcommand*{\subpr}{<}

\def\reals{\mathbb R}
\def\complex{\mathbb C}
\def\integers{\mathbb Z}
\def\rationals{\mathbb Q}
\def\field{\mathbb F}

\newcommand{\lt}{\left}
\newcommand{\rt}{\right}

\newcommand*{\restr}[2]{#1\big|_{#2}}

\newcommand*{\barshbl}{{\bar{s}_\mathrm{HBL}}}
\newcommand*{\shbl}{{s_\mathrm{HBL}}}
\newcommand*{\shbla}[1]{s_{\mathrm{HBL}, #1 }}
\newcommand*{\Mat}[2]{M_{#1}({#2})}

\DeclareMathOperator{\opDim}{dim}
\newcommand*{\Dim}[1]{\opDim({#1})}

\DeclareMathOperator{\opRank}{rank}
\newcommand*{\Rank}[1]{\opRank({#1})}

\DeclareMathOperator{\opKernel}{ker}
\newcommand*{\Kernel}[1]{\opKernel({#1})}

\newcommand*{\sectn}[1]{Section~{#1}}

\newcommand{\ignore}[1]{}

\numberwithin{equation}{section}

\makeatletter
\def\th@plain{%
  \thm@notefont{}
  \itshape 
}
\def\th@definition{%
  \thm@notefont{}
  \normalfont 
}
\makeatother

\theoremstyle{plain}
\newtheorem{theorem}{Theorem}[section]
\newtheorem{definition}[theorem]{Definition}
\newtheorem{notation}[theorem]{Notation}
\newtheorem{remark}[theorem]{Remark}
\newtheorem{lemma}[theorem]{Lemma}
\newtheorem{proposition}[theorem]{Proposition}

\theoremstyle{definition}

\def\fg{{\mathfrak G}}
\def\fv{{\mathbb V}}

\title{%
On H\"older-Brascamp-Lieb Inequalities\\
for Torsion-Free Discrete Abelian Groups}

\author{%
Michael Christ\thanks{%
Mathematics Dept., Univ.\ of California, Berkeley ({\tt mchrist@berkeley.edu}).},
James Demmel\thanks{%
Computer Science Div.\ and Mathematics Dept., Univ.\ of California, Berkeley ({\tt demmel@cs.berkeley.edu}).},
Nicholas Knight\thanks{%
Computer Science Div., Univ.\ of California, Berkeley ({\tt knight@cs.berkeley.edu}).},
Thomas Scanlon\thanks{%
Mathematics Dept., Univ.\ of California, Berkeley ({\tt scanlon@math.berkeley.edu}).},
and Katherine Yelick\thanks{%
Computer Science Div., Univ.\ of California, Berkeley and Lawrence Berkeley Natl.\ Lab.\ ({\tt yelick@cs.berkeley.edu}).}}
\date{October 6, 2015}

\begin{document}
\maketitle

\begin{abstract}
H\"older-Brascamp-Lieb inequalities provide upper bounds for a class of multilinear expressions, in terms of $L^p$ norms
of the functions involved.  They have been extensively studied for functions defined on Euclidean spaces.
Bennett-Carbery-Christ-Tao have initiated the study of these inequalities for discrete Abelian groups
and, in terms of suitable data, have characterized the set of all tuples of exponents
for which such an inequality holds for specified data, as the convex polyhedron defined by a particular finite set of affine inequalities.
 
In this paper we advance the theory of such inequalities for torsion-free discrete Abelian groups in three respects.
The optimal constant in any such inequality is shown to equal $1$ whenever it is finite.
An algorithm that computes the admissible polyhedron of exponents is developed. 
It is shown that nonetheless, existence of an algorithm that computes the full list of inequalities
in the Bennett-Carbery-Christ-Tao description of the admissible polyhedron for all data,
is equivalent to an affirmative solution of Hilbert's Tenth Problem over the rationals.
That problem remains open.
\end{abstract}

\section{Main Results}

By a H\"older-Brascamp-Lieb inequality is typically meant an inequality of the form
\begin{equation} \label{eq:HBLclassical}
\int_{\reals^d} \prod_{j=1}^m f_j(L_j(x))\,dm(x)
\le C
\prod_{j=1}^m \norm{f_j}_{L^{p_j}(\reals^{d_j})}
\end{equation}
where $L_j:\reals^d\to\reals^{d_j}$ are surjective linear mappings,
the functions $f_j$ are Lebesgue measurable and nonnegative,
$m$ denotes Lebesgue measure on $\reals^d$,
the exponents $p_j$ belong to the natural range $[1,\infty]$,
and $C$ is some finite constant which is independent of the functions $f_j$ but may 
depend on the dimensions $d,d_j$, the exponents $p_j$,  and the mappings $L_j$.
Fundamental examples include H\"older's inequality, Young's convolution inequality, and the Loomis-Whitney inequality.
A substantial literature concerning this class of inequalities has developed. 

In this paper we are concerned with analogous inequalities in which $\reals^d,\reals^{d_j}$ are
replaced by discrete Abelian groups $G,G_j$ respectively, where $G$ is torsion-free. Lebesgue measure is replaced by counting measure.
This situation was studied, without the restriction that $G$ be torsion-free, in \cite{BCCT10}.

We have found these inequalities for $G=\integers^d$ to be relevant to the analysis of lower bounds
for communication in a class of algorithms that arise, or may potentially arise, in computer science.
The present paper, motivated by this connection, serves to develop the underlying analytic theory.
The connection and its applications are explored in a companion paper \cite{CDKSY14B}. 

\begin{definition}
\label{def:groupdatumnew}
An \emph{Abelian group HBL datum} $\fg$ is a $3$--tuple 
\[ \fg=\dtup{G,\dtup{G_j},\dtup{\phi_j}} = \dtup{G,\dtup{G_1,\ldots,G_m},\dtup{\phi_1,\ldots,\phi_m}} \] 
where $G$ and each $G_j$ are finitely generated Abelian groups, $G$ is torsion-free, and each $\phi_j\from G \to G_j$ 
is a group homomorphism. 
\end{definition}

When this notation is employed, it is always implicitly assumed that $j\in\set{1,2,\ldots,m}$.
By a H\"older-Brascamp-Lieb inequality associated to an Abelian group HBL datum, we mean one of the form
\begin{equation} \label{BL-C}
\sum_{x\in G} \prod_{j=1}^m f_j(\phi_j(x)) \le C\prod_{j=1}^m \norm{f_j}_{\ell^{1/s_j}} \qquad  
\text{for all functions } 0 \le f_j \in \ell^{1/s_j}(G_j)
\end{equation}
where each $s=(s_1,\dots,s_m)\in[0,1]^m$,
and where $C<\infty$ may depend on 
$\dtup{G,\dtup{G_j},\dtup{\phi_j}}$
and on $s$, but is independent of the functions $f_j$.
The $\ell^p$ norm is $\norm{f}_{\ell^p(G)} = (\sum_{x\in G} |f(x)|^p)^{1/p}$
for $p<\infty$, while $\norm{f}_{\ell^\infty(G)} = \sup_{x\in G} |f(x)|$.
We will often write $\norm{\cdot}_p$ as shorthand for $\norm{\cdot}_{\ell^p}$.

A certain convex polytope $\PP\tup{G,\tup{G_j},\tup{\phi_j}}\subseteq[0,1]^m$ plays a central role in the theory.
\begin{definition} \label{def:P}
For any Abelian group HBL datum $\fg=\tup{G,\tup{G_j},\tup{\phi_j}}$, 
$\PP(\fg)=\PP\tup{G,\tup{G_j},\tup{\phi_j}}$
denotes the set of all $s\in[0,1]^m$ that satisfy 
\begin{equation} \label{subcriticalhypothesis2}
\Rank{H}\le \sum_{j=1}^m s_j\Rank{\phi_j(H)} \qquad \text{for every subgroup } H \subim G.
\end{equation}
\end{definition}

$\PP(\fg)=\PP\tup{G,\tup{G_j},\tup{\phi_j}}$ is the set of all $s\in\reals^m$ specified by the inequalities 
$0\le s_j\le 1$ for all $j$ together with all inequalities
$\sum_{j=1}^m s_j r_j \ge r$, where $\tup{r,r_1,\ldots,r_m}$ ranges over all elements
of $\set{1,2,\ldots,\Rank{G}} \times \prod_{j=1}^m \set{0,1,\ldots,\Rank{\phi_j(G)}}$ 
for which there exists a finitely generated subgroup $H \subim G$ 
that satisfies $\Rank{H} = r$ and $\Rank{\phi_j(H)}=r_j$ for all $j$.
Although infinitely many candidate subgroups $H$ must potentially be considered in any calculation of $\PP(\fg)$, 
there exists a collection of fewer than $(\Rank{G}+1)^{m+1}$ tuples $\tup{r,r_1,\ldots,r_m}$ 
that suffices to generate all of the inequalities defining $\PP(\fg)$.
Thus $\PP$ is a convex polytope with finitely many extreme points. 
It is equal to the convex hull of this finite set.

The first of our three main results states that the optimal constant $C$ in such an inequality equals $1$.
\begin{theorem} \label{thm:2new}
For any Abelian group HBL datum $\fg=\tup{G,\tup{G_j},\tup{\phi_j}}$ 
and any $s\in \PP(\fg)$, 
\begin{equation} \label{BL}
\sum_{x\in G} \prod_{j=1}^m f_j(\phi_j(x)) \le \prod_{j=1}^m \norm{f_j}_{1/s_j} \qquad  
\text{for all functions } 0 \le f_j \in \ell^{1/s_j}(G_j).
\end{equation}
\end{theorem}

In combination with results previously known, this implies a more comprehensive statement.
\begin{theorem} \label{thm:listoffive}
The following conditions are mutually equivalent, for any Abelian group HBL datum 
$\fg=\tup{G,\tup{G_j},\tup{\phi_j}}$
and any $s\in[0,1]^m$.
\begin{enumerate}
\item $s\in \PP(\fg)$.
\item There exists a constant $C<\infty$ for which the H\"older-Brascamp-Lieb inequality \eqref{BL-C} holds. 
\item There exists $A<\infty$ such that
\begin{equation} \label{BL;sets} |E| \le A\prod_{j=1}^m |\phi_j(E)|^{s_j} \qquad\text{for all nonempty finite sets } E \subseteq G.  \end{equation}
\item
The H\"older-Brascamp-Lieb inequality \eqref{BL-C} holds with $C=1$. 
\item Inequality \eqref{BL;sets} holds with $A=1$.
\end{enumerate}
\end{theorem}

Indeed, the first three of these conditions are proved to be equivalent in \cite{BCCT10}, 
in the more general setting where $G$ is an arbitrary finitely generated Abelian group,
possibly with torsion.
In combination with this equivalence, Theorem~\ref{thm:2new} 
implies the mutual equivalence of all five conditions when $G$ is torsion-free.
\qed

If $G$ has torsion then
the optimal constant is no longer equal to $1$. Its value is determined for all finite
Abelian group HBL data in \cite{christfiniteG2013}. In combination with Theorem~\ref{thm:2new}, this allows determination
of the optimal constant for all finitely generated Abelian group HBL data.

That \eqref{BL} cannot hold for any $C<1$ may be seen by fixing any point $x_0\in G$
and defining $f_j$ to be the indicator function of the singleton set $\{\phi_j(x_0)\}$.
%
%
%

Our second main theorem establishes an algorithm for the computation of 
$\PP(\fg)$. To have such an algorithm is desirable 
for applications \cite{CDKSY14B} to computer science.
While 
$\PP(\fg)$
is defined in terms of a finite list of inequalities, its definition refers to each of the 
subgroups of $G$ in order to specify this list.
Therefore the definition does not directly provide an effective way to calculate
$\PP(\fg)$, unless $G$ has rank $0$ or $1$.

\begin{theorem} \label{thm:decision}
There exists an algorithm that takes as input any Abelian group HBL datum 
$\fg=\tup{G,\tup{G_j},\tup{\phi_j}}$
and returns as output both a list of finitely many linear inequalities over $\ZZ$ that jointly 
specify the associated polytope $\PP(\fg)$, and a list of all extreme points 
of $\PP(\fg)$.
\end{theorem}
This algorithm is specified, and proved to be correct, in \S\ref{sec:computeP}.

The polytope $\PP(\fg)$ is specified in Definition~\ref{def:P} by finitely many inequalities,
all of which lie in a finite set that can be read off directly from the datum;
for any subgroup $H\le G$, 
$(r,r_1,\dots,r_m) = (\Rank{H},\Rank{\phi_1(H)},\dots,\Rank{\phi_m(H)})$ is
an element of $\{0,1,\dots,\Rank{G}\}^{m+1}$.
But only some tuples $(r,r_1,\dots,r_m) \in \{0,1,\dots,\Rank{G}\}^{m+1}$
are actually realized in this way by some subgroup.
An algorithm that determines for each $(r,r_1,\dots,r_m)$ whether there exists a subgroup $H\subim G$
satisfying $\Rank{H}=r$ and $\Rank{\phi_j(H)}=r_j$ for each $j\in\{1,2,\dots,m\}$, would 
accomplish the first task of the algorithm of Theorem~\ref{thm:decision}. The algorithm 
of Theorem~\ref{thm:decision} does not do this. Instead, it produces a finite sublist of 
certain $(H,r,r_1,\dots,r_m)$
satisfying these conditions, along with a proof that the polytope specified 
by all the inequalities \eqref{subcriticalhypothesis2}
specified by this sublist coincides with $\PP$. Any inequalities \eqref{subcriticalhypothesis2} 
missing from the computed sublist are guaranteed to be redundant for the definition of  $\PP(\fg)$.

We have not succeeded in devising an algorithm that computes the entire list of constraints \eqref{subcriticalhypothesis2}
in the definition of $\PP(\fg)$.
Our third main theorem asserts that the existence of such an algorithm is equivalent to an affirmative answer
to a longstanding open question, Hilbert's Tenth Problem over the rationals.

\begin{theorem} \label{thm:hilbert} 
There exists an effective algorithm for computing the set of constraints \eqref{subcriticalhypothesis2} defining $\PP(\fg)$ for any arbitrary Abelian group HBL datum $\fg$, 
if and only if there exists an effective algorithm to decide whether any given system of polynomial equations with 
rational coefficients has a rational solution. 
\end{theorem}

A simply computable alternative description of $\PP(\fg)$ is given by Carlen, Lieb, and Loss \cite{CLL04}
for the case in which each target group $G_j$ has rank equal to one.\footnote{ \cite{CLL04}
is concerned with continuum inequalities \eqref{eq:HBLclassical}, but the description of $\PP(\fg)$ obtained
carries over to the Abelian group context.} 
Theorem~\ref{thm:hilbert} makes it clear that the theory is more complex in the general situation
of arbitrary rank.

In a subsequent paper we will refine the results obtained here.
Consider any Abelian group HBL datum 
$\fg=\tup{G,\tup{G_j},\tup{\phi_j}}$.  
Define
$\LL(\fg)$
to be the smallest collection of subgroups
of $G$ that contains $\Kernel{\phi_j}$ for all indices $j$
and is closed with respect to intersections and formation of sums.
Define
the polytope 
$\PP_* = \PP_*(\fg)$
to be the set of all $s\in[0,1]^m$ that satisfy 
\eqref{subcriticalhypothesis2}
for every subgroup $H\in \LL$.
Then
$\PP_*(\fg) =\PP(\fg)$.
Moreover, whereas the algorithm promised by Theorem~\ref{thm:decision} 
and specified in its proof below requires searching a list of all subgroups of $G$ until halting,
it suffices to search only a list of subspaces belonging to  $\LL$.
These can be generated iteratively from the subspaces $\Kernel{\phi_j}$
by repeated formation of intersections and sums.

\section{Upper bounds} \label{sec:hbl}
In this section we prove Theorem~\ref{thm:2new}.
This entails the development of some auxiliary material. 
We have found it more convenient to work with vector spaces $G\otimes\rationals$ than with 
general torsion-free Abelian groups. Therefore we formulate a version of the HBL inequalities in terms
of vector spaces, and reduce Theorem~\ref{thm:2new} to its analogue in the vector space context.

\begin{notation}
Let $X$ be any set.
For $p\in[1,\infty)$, $\ell^p(X)$ denotes the space of all $p$--power summable functions $f\from X\to\CC$, equipped with the norm 
\[ \norm{f}_p = \lt(\sum_{x\in X} |f(x)|^p \rt)^{1/p}.\]
$\ell^\infty(X)$ is the space of all bounded functions $f\from X\to\CC$, equipped with the norm $\norm{f}_\infty = \sup_{x\in X} |f(x)|$.
\end{notation}

These definitions apply even when $X$ is uncountable: If $p\in[1,\infty)$, any function in $\ell^p(X)$ vanishes at all 
but countably many points in $X$; and the $\ell^\infty$--norm is still the supremum of $|f|$. 
In the discussion below, $X$ will always be either a finitely generated Abelian group, or a finite-dimensional vector space over a field $\FF$.
If $s_j=0$ then $1/s_j$ is interpreted as  $\infty$.

\subsection{Preliminaries}
%
%

The following result demonstrates there is no loss of generality in restricting attention to exponent tuples
$s\in[0,1]^m$, rather than $s\in[0,\infty)^m$, in Theorem~\ref{thm:2new}. Its proof is simple, but is deferred to \S\ref{sec:P}
where it is discussed in conjunction with a parallel result in which Abelian groups are replaced by vector spaces.
\begin{proposition} \label{prop:maxexponent}
Let $s\in[0,\infty)^m$.
Define $t\in[0,1]^m$ by $t_j = \min(s_j,1)$ for each index $j$.
If \eqref{subcriticalhypothesis2} or \eqref{BL} holds for $s$ then it holds for $t$.
\end{proposition}

If $X$ is a measure space equipped with counting measure then whenever $0<\alpha\le\beta\le\infty$, 
$\ell^\alpha(X)\subseteq \ell^\beta(X)$, and the inclusion mapping is a contraction. Therefore when $s,t$ are related as above,
inequality \eqref{BL} for the exponent tuple $t$ subsumes the inequality for $s$.

\begin{proof}[Proof of necessity of the hypothesis \eqref{subcriticalhypothesis2}
for validity of the inequality \eqref{BL}]
This necessity is proved in \cite[Theorem~2.4]{BCCT10}; the simple proof is included here for the sake of completeness.
Indeed, consider any subgroup $H \subim G$. 
Let $r=\Rank{H}$. 
By definition of rank, there exists a set $\set{e_i}_{i=1}^r$ of elements of $H$ such that for any coefficients $m_i,n_i\in\ZZ$, $\sum_{i=1}^r m_i e_i = \sum_{i=1}^r n_ie_i$ only if $m_i=n_i$ for all $i$.
For any positive integer $N$ define $E_N$ to be the set of all elements of the form $\sum_{i=1}^r n_i e_i$, where each $n_i$ ranges freely over $\set{1,2,\ldots,N}$. 
Then $|E_N| = N^r$.

On the other hand, for $j\in\set{1,\ldots,m}$,
\begin{equation}\label{necessitycount} 
|\phi_j(E_N)|\le A_j N^{\Rank{\phi_j(H)}}, 
\end{equation}
where $A_j$ is a finite constant which depends on $\phi_j$, on the structure of $H_j$, and on the choice of $\set{e_i}$, but not on $N$.
Indeed, it follows from the definition of rank that for each $j$ it is possible to permute the indices $i$ so that for each $i>\Rank{\phi_j(H)}$ there exist integers $k_i$ and $\kappa_{i,l}$ such that 
\[
k_i \phi_j(e_i)  = \sum_{l=1}^{\Rank{\phi_j(H)}} \kappa_{i,l} \phi_j(e_l).
\]
The upper bound \eqref{necessitycount} follows from these relations.

Consider any large positive integer $N$. 
Define $f_j$ to be the indicator function of $\phi_j(E_N)$.
Then $\prod_j f_j\circ\phi_j\equiv 1$ on $E_N$. Therefore
inequality \eqref{BL}, applied to these functions $f_j$, asserts that 
\[ N^{\Rank{H}} \le \prod_{j=1}^m A_j^{s_j} N^{s_j\Rank{\phi_j(H)}} = AN^{\sum_{j=1}^m s_j\Rank{\phi_j(H)}} \]
where $A<\infty$ is independent of $N$.
By letting $N$ tend to infinity, we conclude that $\Rank{H} \le {\sum_{j=1}^m s_j\Rank{\phi_j(H)}}$, as was to be shown.
\end{proof}

We show that Theorem~\ref{thm:2new} is a consequence of the special case in which all of the groups $G_j$ are torsion-free. 

\begin{proof}[Reduction of Theorem~\ref{thm:2new} to the case of torsion-free codomains]
Let 
$\fg=\tup{G,\tup{G_j},\tup{\phi_j}}$ be an Abelian group HBL datum
and let $s\in\PP(\fg)$.
%
According to the structure theorem of finitely generated Abelian groups, each $G_j$ is isomorphic to $\tilde G_j \oplus T_j$ where $T_j$ is a finite group and $\tilde G_j$ is torsion-free.
Here $\oplus$ denotes the direct sum of Abelian groups; $G'\oplus G''$ is the Cartesian product $G'\times G''$ equipped with the natural componentwise group law.
Define $\pi_j \from \tilde G_j\oplus T_j\to\tilde G_j$ to be the natural projection; thus $\pi_j(x,t)=x$ for $\tup{x,t} \in \tilde G_j \times T_j$.
Define $\tilde \phi_j = \pi_j\circ\phi_j \from G \to \tilde G_j$.

If $K$ is a subgroup of $G_j$, then $\Rank{\pi_j(K)} = \Rank{K}$ since the kernel $T_j$ of $\pi_j$ is a finite group.
Therefore for any subgroup $H\subim G$, $\Rank{\tilde \phi_j(H)}=\Rank{\phi_j(H)}$.  
Therefore whenever 
$\tup{G,\tup{G_j},\tup{\phi_j}}$ 
and $s$ together satisfy the hypothesis \eqref{subcriticalhypothesis2},
so do $\tup{G,\tup{\tilde G_j},\tup{\tilde \phi_j}}$ and $s$.

Under this hypothesis, consider any $m$--tuple $f=\tup{f_j}$ of nonnegative functions with $f_j\in\ell^{1/s_j}(G_j)$, 
and for each $j$, define $\tilde f_j\in \ell^{1/s_j}(\tilde G_j)$ by
\[ \tilde f_j(y) = \max_{t\in T_j} f_j(y,t), \text{ for } y\in \tilde G_j.  \]
For any $x\in G$, $f_j(\phi_j(x))\le \tilde f_j(\tilde \phi_j(x))$.
Consequently $\prod_{j=1}^m f_j(\phi_j(x))\le\prod_{j=1}^m \tilde f_j(\tilde \phi_j(x))$.

We are assuming validity of Theorem~\ref{thm:2new} in the torsion-free case. 
Its conclusion asserts that 
\[ \sum_{x\in G} \prod_{j=1}^m \tilde f_j(\tilde \phi_j(x)) \le \prod_{j=1}^m \norm{\tilde f_j}_{1/s_j}. \]
For each $j$, $\norm{\tilde f_j}_{1/s_j} \le \norm{f_j}_{1/s_j}$. 
This is obvious when $s_j=0$. 
If $s_j \ne 0$ then
\begin{align*}
\norm{\tilde f_j}_{1/s_j}^{1/s_j}
=\sum_{y\in \tilde G_j} \tilde f_j(y)^{1/s_j}
= \sum_{y\in \tilde G_j} \max_{t\in T_j}f_j(y,t)^{1/s_j}
\le \sum_{y\in \tilde G_j} \sum_{t\in T_j}f_j(y,t)^{1/s_j}
= \norm{f_j}_{1/s_j}^{1/s_j}.
\end{align*}
Combining these inequalities gives the HBL inequality \eqref{BL}.
%
\end{proof}

\subsection{Vector space HBL inequalities}
\label{sec:hbl-generalizations}

In this subsection we formulate HBL inequalities for vector spaces over arbitrary fields. 
This vector space framework is more convenient for our method of proof, which involves quotients; 
these are awkward in the Abelian group context since quotients of torsion-free groups need not be torsion-free,
while quotients of vector spaces suffer no such calamity.
We continue to endow all measure spaces with counting measure, even vector spaces over $\reals$ or $\complex$.
Theorem~\ref{thm:field}, below, is the analogue of Theorem~\ref{thm:2new} for vector spaces.
%
%
We will show how Theorem~\ref{thm:field} for $\field=\QQ$ implies Theorem~\ref{thm:2new}, and 
then prove Theorem~\ref{thm:field}.
The case of fields other than $\QQ$ is not the real thrust of our investigation, 
but since it is treated by our analysis with no extra effort we have formulated the
results for general fields.

All vector spaces considered in this paper are assumed to be finite-dimensional.
\begin{notation}
$\Dim{V}$ will denote the dimension of a vector space $V$ over a field $\FF$. 
The notation $W \subim V$ indicates that $W$ is a subspace of $V$, and $W \subpr V$ indicates that $W$ is a proper subspace.
\end{notation}
\begin{definition}
\label{def:vectorspacedatum}
A \emph{vector space HBL datum} is a $3$--tuple 
\[\fv= \dtup{V,\dtup{V_j},\dtup{\phi_j}} = \dtup{V,\dtup{V_1,\ldots,V_m},\dtup{\phi_1,\ldots,\phi_m}} \] 
where $V$ and $V_j$ are finite-dimensional vector spaces over a field $\FF$ 
and $\phi_j\from V \to V_j$ is an $\FF$--linear map. 
\end{definition}
It will always be understood that $j\in\set{1,2,\ldots,m}$, although the quantity $m$ will not usually be indicated.

To any Abelian group HBL datum $\tup{G,\tup{G_j},\tup{\phi_j}}$ with all $G_j$ (as well as $G$) torsion-free 
we associate a vector space HBL datum, as follows.
Firstly,
$G_j$ is isomorphic to $\ZZ^{d_j}$ where $d_j=\Rank{G_j}$. 
Likewise $G$ is isomorphic to $\ZZ^d$ where $d=\Rank{G}$.
Define homomorphisms $\tilde\phi_j\from\ZZ^d\to\ZZ^{d_j}$ by composing $\phi_j$ with these isomorphisms.
$\tup{G,\tup{G_j},\tup{\phi_j}}$ is thereby identified with $\tup{\ZZ^d,\tup{\ZZ^{d_j}},\tup{\tilde \phi_j}}$, 
Defining scalar multiplication in the natural way (i.e., treating $\ZZ^d$ and and $\ZZ^{d_j}$ as $\ZZ$--modules), 
represent each $\ZZ$--linear map $\tilde\phi_j$ by a matrix with integer entries.
Secondly, let $\FF=\QQ$.
Regard $\ZZ^d$ and $\ZZ^{d_j}$ as subsets of $\QQ^d$ and of $\QQ^{d_j}$, respectively, via the inclusion of $\ZZ$ into $\QQ$. 
Define the vector space HBL datum $\tup{V,\tup{V_j},\tup{\psi_j}}$ by setting
$V=\QQ^d$, $V_j = \QQ^{d_j}$, and letting $\psi_j\from\QQ^d \to\QQ^{d_j}$ 
be the $\QQ$--linear map represented by the same integer matrix as $\tilde \phi_j$.

This construction requires certain arbitrary choices, so 
$\tup{V,\tup{V_j},\tup{\psi_j}}$ is properly referred to as {\em an} associated vector space HBL datum.

%
%

%
\begin{theorem} \label{thm:field}
Let $\tup{V,\tup{V_j},\tup{\phi_j}}$ be a vector space HBL datum, and let $s\in[0,1]^m$.  If
\begin{equation} \label{subcriticalhypothesisfield} 
\Dim{W} \le \sum_{j=1}^m s_j\Dim{\phi_j(W)} \qquad \text{for all subspaces } W \subim V
\end{equation}
then
\begin{equation} \label{BLfield}
\sum_{x\in V} \prod_{j=1}^m f_j(\phi_j(x)) \le\prod_{j=1}^m \norm{f_j}_{1/s_j} \qquad \text{for all functions } 0 \le f_j \in \ell^{1/s_j}(V_j).
\end{equation}
In particular,
\begin{equation} \label{BLfield;sets} 
|E| \le \prod_{j=1}^m |\phi_j(E)|^{s_j} \qquad\text{for all nonempty finite sets } E \subseteq V.
\end{equation}
Conversely, if \eqref{BLfield;sets} holds for $s\in[0,1]^m$, and if $\FF=\QQ$ or if $\FF$ is finite, then $s$ satisfies \eqref{subcriticalhypothesisfield}.
\end{theorem}

The conclusions \eqref{BL} and \eqref{BLfield} remain valid for 
functions $f_j$ without the requirement that the $\ell^{1/s_j}$ norms are finite, 
under the convention that the product $0\cdot\infty$ is interpreted as zero whenever it arises. 
For then if any $f_j$ has infinite norm, then either the right-hand side is infinite, or both sides are zero; in either case, the inequality holds. 

\begin{proof}[Reduction of Theorem~\ref{thm:2new} in the case of torsion-free codomains to Theorem~\ref{thm:field}]
Let $\fg = 
\tup{G,\tup{G_j},\tup{\phi_j}}$ be an Abelian group HBL datum, 
and let $s\in\PP(\fg)$; furthermore suppose that each group $G_j$ is torsion-free. 
Let $\fv=\tup{V,\tup{V_j},\tup{\psi_j}}$ be a vector space HBL datum associated to
$\fg$ by the construction specified above.
We claim that $s\in\PP(\fv)$.
%
Indeed, given any subspace $W \subim V$, there exists a basis $S$ for $W$ over $\QQ$ which consists of elements of $\ZZ^d$. 
Define $H\subim \ZZ^d$ to be the subgroup generated by $S$ (over $\ZZ$).
Then $\Rank{H}=\Dim{W}$.  
Moreover, $\psi_j(W)$ equals the span over $\QQ$ of $\tilde \phi_j(H)$, and $\Dim{\psi_j(W)} = \Rank{\tilde \phi_j(H)}$.
The hypothesis $\Rank{H}\le \sum_{j=1}^m s_j \Rank{\tilde \phi_j(H)}$ is therefore equivalently written as $\Dim{W} \le \sum_{j=1}^m s_j \Dim{\psi_j(W)}$, which is \eqref{subcriticalhypothesisfield} for $W$.
Since this holds for every subspace $W$, $s\in\PP(\fv)$.

Consider the $m$--tuple $f=\tup{f_j}$ corresponding to any inequality in \eqref{BL} for $\tup{\ZZ^d,\tup{\ZZ^{d_j}},\tup{\tilde \phi_j}}$ and $s$, and for each $j$ define $F_j$ by $F_j(y)=f_j(y)$ if $y\in\ZZ^{d_j}\subseteq\QQ^{d_j}$, and $F_j(y)=0$ otherwise.
%
%
By Theorem~\ref{thm:field}, 
\[
\sum_{x\in \ZZ^d} \prod_{j=1}^m f_j(\tilde \phi_j(x)) 
= \sum_{x\in \QQ^d} \prod_{j=1}^m F_j(\psi_j(x)) 
\le \prod_{j=1}^m \norm{F_j}_{1/s_j}
= \prod_{j=1}^m \norm{f_j}_{1/s_j}.
\]
Thus the conclusion \eqref{BL} of Theorem~\ref{thm:2new} is satisfied.
%
%
\end{proof}

Conversely, it is possible to derive Theorem~\ref{thm:field} for $\FF=\QQ$ from the special case of 
Theorem~\ref{thm:2new} in which $G=\ZZ^d$ and $G_j = \ZZ^{d_j}$ by similar reasoning, using 
multiplication by large integers to clear denominators.

This reasoning in this reduction establishes
\begin{lemma} \label{lem:PZequalsPQ}
Let $\tup{G,\tup{G_j},\tup{\phi_j}}$ be an Abelian group HBL datum with torsion-free codomains $G_j$, 
let $\tup{\ZZ^d,\tup{\ZZ^{d_j}},\tup{\tilde \phi_j}}$ be the associated datum specified above, 
and let $\tup{\QQ^d,\tup{\QQ^{d_j}},\tup{\psi_j}}$ be an associated vector space HBL datum.
Then 
\[ \PP\dtup{G,\dtup{G_j},\dtup{\phi_j}} = \PP\dtup{\ZZ^d,\dtup{\ZZ^{d_j}},\dtup{\tilde \phi_j}} = \PP\dtup{\QQ^d,\dtup{\QQ^{d_j}},\dtup{\psi_j}}.  \] 
\end{lemma}

\subsection{On groups $G$ with torsion} \label{subsect:torsion}
%
In this subsection we temporarily relax the requirement in the definition of an Abelian group HBL datum 
that the finitely generated Abelian group $G$ be torsion-free. We call such a more general $(G,(G_j),(\phi_j))$ an HBL datum with torsion.
%

%
%
It was shown in \cite[Theorem~2.4]{BCCT10} that for an HBL datum with torsion, 
the rank conditions \eqref{subcriticalhypothesis2} are necessary and sufficient for the existence of some finite constant $C$ 
such that \eqref{BL-C} holds.  %
A consequence of Theorem~\ref{thm:2new} is a concrete upper bound for the constant $C$ in these inequalities.
The torsion subgroup $T(G)$ of $G$ is the (finite) set of all elements $x\in G$ for which there exists $0\ne n\in\ZZ$ such that $nx=0$.

\begin{theorem}
Consider $\tup{G,\tup{G_j},\tup{\phi_j}}$, an HBL datum with torsion, and $s\in[0,1]^m$.
If \eqref{subcriticalhypothesis2} holds, then \eqref{BL-C} holds with $C=|T(G)|$.
In particular,  
\begin{equation} \label{inequalitywithtorsion}
|E|\le |T(G)|\cdot \prod_{j=1}^m |\phi_j(E)|^{s_j} \qquad\text{for all nonempty finite sets } E \subseteq G.
\end{equation}
Conversely, if \eqref{inequalitywithtorsion} holds for $s\in[0,1]^m$, then $s$ satisfies \eqref{subcriticalhypothesis2}.
\end{theorem}
\begin{proof}
To prove \eqref{BL-C}, express $G$ isomorphically as $\tilde G\oplus T(G)$ where $\tilde G \le G$ is torsion-free.
Thus arbitrary elements $x\in G$ are expressed as $x=(\tilde x,t)$  with $\tilde x\in\tilde G$ and $t\in T(G)$.
Then $\tup{\tilde G,\tup{G_j},\tup{\restr{\phi_j}{\tilde G}}}$ is an Abelian group HBL datum (with $\tilde G$ torsion-free) to which Theorem~\ref{thm:2new} can be applied. 
Consider any $t\in T(G)$. 
Define $g_j\from G_j\to[0,\infty)$ by $g_j(x_j) = f_j(x_j+\phi_j(0,t))$.
Then $f_j(\phi_j(y,t)) = f_j(\phi_j(y,0)+\phi_j(0,t)) = g_j(\phi_j(y,0))$, so
\begin{equation*}
\sum_{y\in\tilde G} \prod_{j=1}^m f_j(\phi_j(y,t))
= \sum_{y\in\tilde G} \prod_{j=1}^m g_j(\phi_j(y,0))
\le \prod_{j=1}^m \norm{\restr{g_j}{\phi_j(\tilde G)}}_{1/s_j}
\le \prod_{j=1}^m \norm{f_j}_{1/s_j}.
\end{equation*}
The first inequality is an application of Theorem~\ref{thm:2new}.
Summation with respect to $t\in T(G)$ gives the required bound.

To show necessity, we consider just the inequalities \eqref{inequalitywithtorsion} corresponding to the subsets $E \subseteq \tilde G$ and follow the proof of the converse of Theorem~\ref{thm:2new}, except substituting $A|T(G)|$ for $A$.
\end{proof}
The factor $|T(G)|$ cannot be improved if the groups $G_j$ are torsion free, or more generally if $T(G)$ is contained in the intersection of the kernels of all the homomorphisms $\phi_j$; this is seen by considering $E=T(G)$. However, it is not optimal, in general. The optimal bound, for arbitrary finitely generated
Abelian groups, is determined in a paper \cite{christfiniteG2013} that builds on the present one.

\subsection{Polytopes $\PP$ for vector space HBL data} \label{sec:P}

\begin{definition} \label{def:P_vector}
For any vector space HBL datum $\fv=\tup{V,\tup{V_j},\tup{\phi_j}}$, 
$\PP(\fv)$ denotes the set of all $s\in[0,1]^m$ that satisfy \eqref{subcriticalhypothesisfield}. 
\end{definition}

Now we prove Proposition~\ref{prop:maxexponent}, which asserts that there is no loss of information in restricting the discussion to exponents $s_j \le 1$.
\begin{proof}[Proof of Proposition~\ref{prop:maxexponent}]
Proposition~\ref{prop:maxexponent} was formulated in terms of Abelian group HBL data and inequalities.
Here we show the corresponding result for vector space HBL data and inequalities, obtaining the version for Abelian groups as a direct consequence.

%
Suppose that a vector space HBL datum $\tup{V,\tup{V_j},\tup{\phi_j}}$ and $s\in[0,\infty)^m$ satisfy \eqref{subcriticalhypothesisfield}, and suppose that $s_k >1$ for some $k$.
Define $t \in [0,\infty)^m$ by $t_j=s_j$ for $j \ne k$, and $t_k=1$.
We claim that $t$ continues to satisfy \eqref{subcriticalhypothesisfield}.

Consider any subspace $W\subim V$. In order to verify \eqref{subcriticalhypothesisfield}
for $W$ and $t$, define $W' = W\cap\Kernel{\phi_k}$. 
Choose a supplement $U$ for $W'$ in $W$;
that is, $W=W'+U$ and $W'\cap U=\{0\}$.
Then by \eqref{subcriticalhypothesisfield} for $s$,
\[ \Dim{W'}\le \sum_{j=1}^m s_j\Dim{\phi_j(W')} 
= s_k\cdot 0 + \sum_{j\ne k} s_j\Dim{\phi_j(W')}
= \sum_{j\ne k} t_j\Dim{\phi_j(W')}.  \]
Since $\phi_k$ is injective on $U$ and $t_k=1$, $\Dim{U} = t_k\Dim{\phi_k(U)}$.
Therefore 
\begin{align*}
\Dim{W}
&= \Dim{U} + \Dim{W'}
\\&= t_k\Dim{\phi_k(U)} + \Dim{W'}
\\&\le t_k\Dim{\phi_k(U)} + \sum_{j\ne k} t_j\Dim{\phi_j(W')}
\\& \le  t_k\Dim{\phi_k(W)} + \sum_{j\ne k} t_j\Dim{\phi_j(W)}
\\& = \sum_{j=1}^m t_j\Dim{\phi_j(W)}.
\end{align*}

Given an $m$--tuple $s$ with multiple components $s_k > 1$, we 
argue by induction on the number of such indices $k$, with the result just proved serving as the induction step.

Our desired conclusion concerning \eqref{subcriticalhypothesis2} follows from Lemma~\ref{lem:PZequalsPQ}, 
by considering the associated vector space HBL datum $\tup{\QQ^d,\tup{\QQ^{d_j}},\tup{\psi_j}}$ and noting that the lemma was established without assuming $s_j \le 1$.

Next, we show the result concerning \eqref{BL}: If \eqref{BL} holds for some $s\in[0,\infty)^m$,
then it also holds for $t$, where $t_j = \min(s_j,1)$ for all $j\in\{1,2,\dots,m\}$. 
Consider the following less structured situation.
Let $X,X_1,\ldots,X_m$ be sets and $\phi_j \from X \to X_j$ be functions for $j\in\set{1,\ldots,m}$. 
Let $s\in[0,\infty)^m$ with some $s_k>1$, and suppose that $|E| \le \prod_{j=1}^m |\phi_j(E)|^{s_j}$ for any finite nonempty subset $E \subseteq X$. 
Consider any such set $E$.
For each $y \in \phi_k(E)$, let $E_y = \phi_k^{-1}(y) \cap E$, the preimage of $y$ under $\restr{\phi_k}{E}$; thus $|\phi_k(E_y)| = 1$.
By assumption, $|E_y| \le \prod_{j=1}^m |\phi_j(E_y)|^{s_j}$, so it follows that 
\[
|E_y| \le \prod_{j\ne k}|\phi_j(E_y)|^{s_j} \le \prod_{j\ne k} |\phi_j(E)|^{s_k} =  \prod_{j\ne k} |\phi_j(E)|^{t_k}.
\]
Since $E$ can be written as the union of disjoint sets $\bigcup_{y\in \phi_k(E)} E_y$, we obtain
\[ 
|E| = \sum_{y\in \phi_k(E)} |E_y| \le \sum_{y \in  \phi_k(E)} \prod_{j \ne k} |\phi_j(E)|^{t_k} = |\phi_k(E)| \cdot \prod_{j \ne k} |\phi_j(E)|^{t_k} =  \prod_{j =1}^m |\phi_j(E)|^{t_k}.
\]
Given an $m$--tuple $s$ with multiple components $s_k > 1$, we can consider each separately and apply the same reasoning to obtain an $m$--tuple $t\in[0,1]^m$ with $t_j = \min(s_j,1)$ for all $j$.
By picking $\tup{X,\tup{X_j},\tup{\phi_j}} \ceq \tup{\ZZ^d,\tup{\ZZ^{d_j}},\tup{\phi_j}}$, 
we now conclude that \eqref{BL;sets} holds for $t$.
By Theorem~\ref{thm:listoffive}, \eqref{BL;sets} implies \eqref{BL};
the proof of that theorem given below is independent of Proposition~\ref{prop:maxexponent} so this concludes its proof.

Similar conclusions can be obtained for \eqref{BL;sets} and \eqref{BLfield;sets}.
\end{proof}

\subsection{Interpolation between extreme points of $\PP$} \label{subsect:interpolation}

To complete the proof of Theorem~\ref{thm:field}, we must show that \eqref{BLfield} holds for every $s\in\PP$.
We next prove that if \eqref{BLfield} holds at each extreme point $s$ of $\PP$, then it holds for all $s\in\PP$.
In \sectn{\ref{subsect:criticalsubspaces}}, we will show that for any extreme point $s$ of $\PP$, 
the hypothesis \eqref{subcriticalhypothesisfield} can be restated in a special form.
Finally, in \sectn{\ref{sec:thm:field-proof}} we will use this reformulation to prove \eqref{BLfield} for extreme points.

%
%
Let $\tup{X_j,\mathcal{A}_j,\mu_j}$ be measure spaces for $j\in\set{1,2,\ldots,m}$, where each $\mu_j$ is a nonnegative measure on the $\sigma$--algebra $\mathcal{A}_j$.
Let $S_j$ be the set of all simple functions $f_j \from X_j\to\CC$.
Thus $S_j$ is the set of all $f_j \from X_j\to\CC$ that can be expressed in the form $\sum_i c_i \indicator_{E_i}$ where $c_i\in\CC$, $E_i\in\mathcal{A}_j$, $\mu_j(E_i)<\infty$, and the sum extends over finitely many indices $i$.

Let $T \from \prod_{j=1}^m \CC^{X_j} \to \CC$ be a multilinear map; i.e., for any $m$--tuple $f\in \prod_{j=1}^m \CC^{X_j}$ where $f_k=c_0f_{k,0}+c_1f_{k,1}$ for $c_0,c_1\in\CC$ and $f_{k,0},f_{k,1}\in\CC^{X_k}$,
%
\[
T(f)
= c_0T\dtup{f_1,\ldots,f_{k-1},f_{k,0},f_{k+1},\ldots,f_m} + c_1T\dtup{f_1,\ldots,f_{k-1},f_{k,1},f_{k+1},\ldots,f_m}.
\]
%

One multilinear extension of the Riesz-Thorin theorem states the following (see, e.g., \cite{bennettsharpley}).
\begin{proposition}[Multilinear Riesz--Thorin theorem]
\label{prop:rieszthorin}
Suppose that $p_0=\tup{p_{j,0}},p_1=\tup{p_{j,1}} \in [1,\infty]^m$. 
Suppose that there exist $A_0,A_1\in[0,\infty)$ such that
\begin{equation*}
\lt|T(f)\rt| \le A_0\prod_{j=1}^m \norm{f_j}_{p_{j,0}} \qquad\text{and}\qquad \lt|T(f)\rt| \le A_1\prod_{j=1}^m \norm{f_j}_{p_{j,1}} \qquad\text{for all } f \in \prod_{j=1}^m S_j.
\end{equation*}
For each $\theta\in(0,1)$ define exponents $p_{j,\theta}$ by
\begin{equation*}
\frac1{p_{j,\theta}} \ceq \frac{\theta}{p_{j,0}} + \frac{1-\theta}{p_{j,1}}.
\end{equation*}
Then for each $\theta\in(0,1)$,
\begin{equation*}
 \lt|T(f)\rt| \le A_0^\theta A_1^{1-\theta}\prod_{j=1}^m \norm{f_j}_{p_{j,\theta}} \qquad\text{for all } f \in \prod_{j=1}^m S_j.
 \end{equation*}
Here $\norm{f_j}_{p}=\norm{f_j}_{L^{p}(X_j,\mathcal{A}_j,\mu_j)}$.
\end{proposition}

In the context of Theorem~\ref{thm:field} with  vector space HBL datum $\tup{V,\tup{V_j},\tup{\phi_j}}$, we consider the multilinear map 
\[ T(f) \ceq \sum_{x\in V} \prod_{j=1}^m f_j(\phi_j(x)) \]
representing the left-hand side in \eqref{BLfield}.
\begin{lemma}
\label{lem:atextremepoints}
If \eqref{BLfield} holds for every extreme point of $\PP$, then it holds for every $s\in\PP$.
\end{lemma}
\begin{proof}
For any $\tilde f\in \prod_{j=1}^m S_j$, we define another $m$--tuple $f$ where for each $j$,  $f_j = |\tilde f_j|$ is a nonnegative simple function.
By hypothesis, the inequality in \eqref{BLfield} corresponding to $f$ holds at every extreme point $s$ of $\PP$, giving
\[
\lt| T(\tilde f)\rt| 
\le \sum_{x\in V} \prod_{j=1}^m \lt|\tilde f_j(\phi_j(x))\rt| 
= \sum_{x\in V} \prod_{j=1}^m f_j(\phi_j(x)) 
\le \prod_{j=1}^m \norm{f_j}_{1/s_j} 
= \prod_{j=1}^m \norm{\tilde f_j}_{1/s_j}.
\]
As a consequence of Proposition~\ref{prop:rieszthorin} (with constants $A_i=1$), and the fact that any $s\in\PP$ is a finite convex combination of the extreme points, this expression holds for any $s\in\PP$.
For any nonnegative function $F_j$ (e.g., in $\ell^{1/s_j}(V_j)$), there is an increasing sequence of nonnegative simple functions $f_j$ whose (pointwise) limit is $F_j$. 
Consider the $m$--tuple $F=\tup{F_j}$ corresponding to any inequality in \eqref{BLfield}, and consider a sequence of $m$--tuples $f$ which converge to $F$; then $\prod_{j=1}^m f_j$ also converges to $\prod_{j=1}^m F_j$. 
So by the monotone convergence theorem, the summations on both sides of the inequality converge as well.
\end{proof}

\subsection{Critical subspaces and extreme points} 
\label{subsect:criticalsubspaces}
Let $\fv=\tup{V,\tup{V_j},\tup{\phi_j}}$ be an arbitrary vector space HBL datum. 
Let $\PP=\PP(\fv)$ continue to denote the set of all $s\in[0,1]^m$ 
that satisfy \eqref{subcriticalhypothesisfield}.
The following key definition appears in \cite{BCCT08} and \cite{BCCT10}.
\begin{definition}
Consider any $s\in[0,1]^m$.
A subspace $W \subim V$ satisfying $\Dim{W}=\sum_{j=1}^m s_j\Dim{\phi_j(W)}$ is said to be a critical subspace with respect to $s$; 
one satisfying $\Dim{W}\le \sum_{j=1}^m s_j\Dim{\phi_j(W)}$ is said to be subcritical with respect to $s$; 
and a subspace satisfying $\Dim{W} > \sum_{j=1}^m s_j\Dim{\phi_j(W)}$ is said to be supercritical with respect to $s$.
$W$ is said to be strictly subcritical with respect to $s$ if $\Dim{W}<\sum_{j=1}^m s_j\Dim{\phi_j(W)}$.
\end{definition}
In this language, the conditions \eqref{subcriticalhypothesisfield} assert that that every subspace $W$ of $V$, including $\set{0}$ and $V$ itself, is subcritical; equivalently, there are no supercritical subspaces.  
We may sometimes omit the phrase ``with respect to $s$'', but these notions are always relative to some implicitly or explicitly specified tuple $s$ of exponents.

The goal of \sectn{\ref{subsect:criticalsubspaces}} is to establish the following:
\begin{proposition} 
\label{prop:0or1}
Let $s$ be an extreme point of $\PP(\fv)$.
Then some subspace $\set{0} \subpr W \subpr V$ is critical with respect to $s$, or $s\in\set{0,1}^m$.
\end{proposition}
\noindent These two possibilities are not mutually exclusive.
%
%
\begin{lemma} 
\label{lemma:0}
If $s$ is an extreme point of $\PP(\fv)$, and if $i$ is an index for which $s_i\notin\set{0,1}$, then $\Dim{\phi_i(V)}\ne 0$.
\end{lemma}
\begin{proof}
Suppose $\Dim{\phi_i(V)}=0$.
If $t\in[0,1]^m$ satisfies $t_j=s_j$ for all $j\ne i$, then $\sum_{j=1}^m t_j\Dim{\phi_j(W)} = \sum_{j=1}^m s_j \Dim{\phi_j(W)}$ for all subspaces $W \subim V$, so $t\in\PP(\fv)$ as well.
If $s_i\notin\set{0,1}$, then this contradicts the assumption that $s$ is an extreme point of $\PP(\fv)$. 
\end{proof}
\begin{lemma} 
\label{lemma:2}
Let $s$ be an extreme point of $\PP(\fv)$.
Suppose that no subspace $\set{0} \subpr W \subim V$ is critical with respect to $s$. 
Then $s\in\set{0,1}^m$.
\end{lemma}
\begin{proof}
Membership in $\PP(\fv)$ is decided by finitely many affine inequalities. The hypothesis of the lemma
means that $s$ satisfies each of these with strict inequality, except for the
inequality $0\le \sum_j s_j\cdot 0$ that arises from $W=\{0\}$.
Consequently any $t\in[0,1]^m$ that is sufficiently close to $s$ also belongs to $\PP(\fv)$.
If $s_i\notin\set{0,1}$ for some index $i$, then any 
$t\in[0,1]^m$ with $t_j=s_j$ for all $j\ne i$ and $t_i$ sufficiently close to $s_i$ belongs to $t\in\PP(\fv)$. 
This contradicts the assumption that $s$ is an extreme point.
\end{proof}
\begin{lemma} 
\label{lemma:1}
Let $s$ be an extreme point of $\PP(\fv)$.
Suppose that no subspace $\set{0} \subpr W \subpr V$ is critical with respect to $s$.
Then there exists at most one index $i$ for which $s_i\notin\set{0,1}$.
\end{lemma}
\begin{proof}
Suppose to the contrary that there were to exist distinct indices $k,l$ such that neither of $s_k,s_l$ belongs to $\set{0,1}$.
By Lemma~\ref{lemma:0}, both $\phi_k(V)$ and $\phi_l(V)$ have positive dimensions.
For $\eps\in\RR$ define $t$ by $t_j=s_j$ for all $j\notin\set{k,l}$,
\begin{equation*} 
t_k = s_k+\eps \Dim{\phi_l(V)} \text{ and } t_l = s_l-\eps \Dim{\phi_k(V)}.  
\end{equation*}
Whenever $|\eps|$ is sufficiently small, $t\in[0,1]^m$.
Moreover, $V$ remains subcritical with respect to $t$. 
If $|\eps|$ is sufficiently small, then every subspace $\set{0} \subpr W \subpr V$ remains strictly subcritical with respect to $t$, because the set of all parameters $\tup{\Dim{W},\Dim{\phi_1(W)},\ldots,\Dim{\phi_m(W)}}$ which arise, is finite. 
Thus $t\in\PP(\fv)$ for all sufficiently small $|\eps|$. Therefore $s$ is not an extreme point of $\PP(\fv)$.
\end{proof}
\begin{lemma} 
\label{lemma:3}
Let $s\in[0,1]^m$ be an extreme point of $\PP(\fv)$. Suppose that $V$ is critical with respect to $s$.
Suppose that there exists exactly one index $i\in \set{1,2,\ldots,m}$ for which $s_i\notin\set{0,1}$. 
Then $V$ has a subspace that is supercritical with respect to $s$.
\end{lemma}
\begin{proof}
By Lemma~\ref{lemma:0}, $\Dim{\phi_i(V)}>0$.
Let $K$ be the set of all indices $k$ for which $s_k=1$.
The hypothesis that $V$ is critical means that
\begin{equation*}
\Dim{V} = s_i\Dim{\phi_i(V)} + \sum_{k\in K} s_k\Dim{\phi_k(V)}. 
\end{equation*}
Since $s_i>0$ and $\Dim{\phi_i(V)}>0$,
\begin{equation*}
\sum_{k\in K}\Dim{\phi_k(V)} = \sum_{k\in K} s_k \Dim{\phi_k(V)} = \Dim{V}-s_i\Dim{\phi_i(V)} < \Dim{V}.
\end{equation*}
Consider the subspace $W \subim V$ defined by
\begin{equation*} 
W=\bigcap_{k\in K}\Kernel{\phi_k};
\end{equation*}
this intersection is interpreted to be $W=V$ if the index set $K$ is empty.
$W$ necessarily has positive dimension. 
Indeed, $W$ is the kernel of the map $\psi\from V\to\bigoplus_{k\in K}\phi_k(V)$, defined by $\psi(x)=\tup{\phi_k(x): k\in K}$, where $\bigoplus$ denotes the direct sum of vector spaces.
The image of $\psi$ is a subspace of $\bigoplus_{k\in K}\phi_k(V)$, a vector space whose dimension $\sum_{k\in K} \Dim{\phi_k(V)}$ is strictly less than $\Dim{V}$.
Therefore $\Kernel{\psi}=W$ has dimension greater than or equal to $\Dim{V}-\sum_{k\in K}\Dim{\phi_k(V)}>0$.
Since $\phi_k(W)=\set{0}$ for all $k\in K$,
\begin{align*} 
\sum_{j=1}^m s_j \Dim{\phi_j(W)} &= s_i\Dim{\phi_i(W)} + \sum_{k\in K} \Dim{\phi_k(W)} \\
&= s_i\Dim{\phi_i(W)} \\
&\le s_i\Dim{W}.
\end{align*}
Since $s_i<1$  and $\Dim{W}>0$, $s_i\Dim{\phi_i(W)}$ is strictly less than $\Dim{W}$, whence $W$ is supercritical. 
\end{proof}
\begin{proof}[Proof of Proposition~\ref{prop:0or1}]
Suppose that there exists no critical subspace $\set{0} \subpr W \subpr V$. 
By Lemma~\ref{lemma:2}, either $s\in\set{0,1}^m$ --- in which case the proof is complete --- or $V$ is critical.
By Lemma~\ref{lemma:1}, there can be at most one index $i$ for which $s_i\notin\set{0,1}$.
By Lemma~\ref{lemma:3}, for critical $V$, the existence of one single such index $i$ implies the presence of some supercritical subspace, contradicting the main hypothesis of Proposition~\ref{prop:0or1}.  
Thus again, $s\in\set{0,1}^m$. 
\end{proof}

\subsection{Factorization of HBL data}
\begin{notation}
\label{not:splittingspaces}
Suppose $V,V'$ are finite-dimensional vector spaces over a field $\FF$, and $\phi\from V\to V'$ is an $\FF$--linear map.
For any subspace $W \subim V$, $\restr{\phi}{W}\from W \to \phi(W)$ denotes the restriction of $\phi$ to $W$, also a $\FF$--linear map.
$V/W$ denotes the quotient of $V$ by $W$, a finite-dimensional vector space. 
%
Thus $x+W=x'+W$ if and only if $x-x'\in W$. 
%
Every subspace of $V/W$ can be written as $U/W$ for some $W \subim U \subim V$.

The quotient space $V'/\phi(W)$, and the quotient linear map $[\phi]\from V/W \ni x+W \mapsto \phi(x) + \phi(W) \in V'/\phi(W)$,
are likewise defined. 
\end{notation}
%
%

In this quotient situation,
it is elementary that for any $U/W \subim V/W$, $[\phi](U/W)=\phi(U)/\phi(W)$.

\begin{definition}
Let $\tup{V,\tup{V_j},\tup{\phi_j}}$ be a vector space HBL datum.
To any subspace $W\subim V$ are associated the two vector space HBL data
\begin{equation}
\left\{
\begin{aligned} \fv_W &=\tup{W,\tup{\phi_j(W)},\tup{\restr{\phi_j}{W}}} 
\\ \fv_{V/W} &= \tup{V/W,\tup{V_j/\phi_j(W)},\tup{[\phi_j]}}. 
\end{aligned} \right. \end{equation}
\end{definition}

\begin{lemma} \label{lemma:factorization1}
Let $\fv=\tup{V,\tup{V_j},\tup{\phi_j}}$ be  a vector space HBL datum. For any subspace $W\subim V$, 
%
\begin{equation}
\PP(\fv_W)\,\cap\,\PP(\fv_{V/W})
\subseteq \PP(\fv).
\end{equation}
\end{lemma}
\begin{proof}
Consider any subspace $U\subim V$ and any $s\in \PP(\fv_W) \cap \PP(\fv_{V/W})$.
%
Then
\begin{align*}
\Dim{U} &= \Dim{(U+W)/W} + \Dim{U\cap W} \\
&\le \sum_{j=1}^m s_j \Dim{[\phi_j]((U+W)/W))}  + \sum_{j=1}^m s_j \Dim{\phi_j(U\cap W)} \\
&= \sum_{j=1}^m s_j \Dim{\phi_j(U+W)/\phi_j(W)}  + \sum_{j=1}^m s_j \Dim{\phi_j(U\cap W)} \\
&= \sum_{j=1}^m s_j \lt(\Dim{\phi_j(U+W)} - \Dim{\phi_j(W)} \rt) + \sum_{j=1}^m s_j \Dim{\phi_j(U\cap W)} \\
&= \sum_{j=1}^m s_j \lt(\Dim{\phi_j(U)+\phi_j(W)} + \Dim{\phi_j(U\cap W)} - \Dim{\phi_j(W)}\rt) \\
&\le \sum_{j=1}^m s_j \lt(\Dim{\phi_j(U)+\phi_j(W)} + \Dim{\phi_j(U)\cap \phi_j(W)} - \Dim{\phi_j(W)}\rt) \\
&=  \sum_{j=1}^m s_j \Dim{\phi_j(U)}.
\end{align*}
The last inequality is a consequence of the inclusions
%
%
$\phi_j(U\cap W) \subseteq \phi_j(U)\cap \phi_j(W)$. 
The last equality uses
the relation $\Dim{A}+\Dim{B} = \Dim{A+B}+\Dim{A\cap B}$, which holds for any subspaces $A,B$ of a vector space.
Thus $U$ is subcritical with respect to $s$, so $s\in\PP(\fv)$.
\end{proof}
\begin{lemma} \label{lemma:factorization}
Let $\fv=\tup{V,\tup{V_j},\tup{\phi_j}}$ be a vector space HBL datum. 
Let $s\in[0,1]^m$ and let $W\subim V$ be a subspace of $V$.
%
If $W$ is critical with respect to $s$ then
\begin{equation} s\in\PP(\fv) \ \Longleftrightarrow\ s\in  \PP(\fv_W)\,\cap\,\PP(\fv_{V/W}).  \end{equation}
%
\end{lemma}
\begin{proof}
With Lemma~\ref{lemma:factorization1} in hand, it remains to show that 
if $s\in\PP(\fv)$ then  $ s\in  \PP(\fv_W)\,\cap\,\PP(\fv_{V/W})$.
Any subspace $U \subim W$ is also a subspace of $V$.
$U$ is subcritical with respect to $s$ when regarded as a subspace of $W$, if and only if $U$ is subcritical when regarded as a subspace of $V$.
So $s\in\PP(\fv_W)$.

Next consider any subspace $U/W \subim V/W$. We have $W \subim U \subim V$ and $\Dim{U/W} = \Dim{U}-\Dim{W}$.
%
%
Moreover, 
\[ \Dim{[\phi_j](U/W)} = \Dim{\phi_j(U)/\phi_j(W)} = \Dim{\phi_j(U)}-\Dim{\phi_j(W)}.  \]
Therefore since $\Dim{W} = \sum_{j=1}^m s_j\Dim{\phi_j(W)}$,
\begin{multline*}
\Dim{U/W} = \Dim{U}-\Dim{W}
\le \sum_{j=1}^m s_j \Dim{\phi_j(U)} - \sum_{j=1}^m s_j\Dim{\phi_j(W)}\\ 
= \sum_{j=1}^m s_j \lt(\Dim{\phi_j(U)} - \Dim{\phi_j(W)}\rt)
= \sum_{j=1}^m s_j \Dim{[\phi_j](U/W)}
\end{multline*}
by the subcriticality of $U$, which holds because $s\in\PP(\fv)$.
Thus any subspace $U/W \subim V/W$ is subcritical with respect to $s$, so $s\in \PP(\fv_{V/W})$, as well.
\end{proof}

\subsection{Conclusion of proof of Theorem~\ref{thm:field}} \label{sec:thm:field-proof}
In order to prove Theorem~\ref{thm:field} for a vector space HBL datum $\tup{V,\tup{V_j},\tup{\phi_j}}$, we argue
by induction on the dimension of the ambient vector space $V$.
If $\Dim{V}=0$ then $V$ has a single element, and the inequality \eqref{BLfield} is trivially valid.

To establish the inductive step, consider any extreme point $s$ of $\PP$.
According to Proposition~\ref{prop:0or1}, there are two cases which must be analyzed.
We begin with the case in which there exists a critical subspace $\set{0} \subpr W \subpr V$. 
We assume that Theorem~\ref{thm:field} holds for all HBL data for which the ambient vector space has strictly smaller dimension than is the case for the given datum.

\begin{lemma} \label{lemma:criticalW}
Let $\fv=\tup{V,\tup{V_j},\tup{\phi_j}}$ be a vector space HBL datum, and let $s\in\PP(\fv)$.
Suppose that there exists a subspace $\set{0} \subpr W \subpr V$ that is critical with respect to $s$.
Then the inequality \eqref{BLfield} holds for this $s$.
\end{lemma}
\begin{proof}
Consider the inequality \eqref{BLfield} for some $s\in[0,1]^m$.
We may assume that none of the exponents $s_j$ equals zero. 
For if $s_k=0$, then $f_k(\phi_k(x)) \le \norm{f_k}_{1/s_k}$ for all $x$, and therefore 
\[\sum_{x\in V} \prod_{j=1}^m f_j(\phi_j(x)) \le \norm{f_k}_{1/s_k} \cdot \sum_{x\in V} \prod_{j\ne k} f_j(\phi_j(x)).\] 
If $\norm{f_k}_{1/s_k}=0$, then \eqref{BLfield} holds with both sides $0$. 
Otherwise we divide by $\norm{f_k}_{1/s_k}$ to conclude that $s\in\PP(\fv)$ if and only if $\tup{s_j}_{j \ne k}$ 
belongs to the polytope associated to the vector space HBL datum $\tup{V,\tup{V_j}_{j\ne k},\tup{\phi_j}_{j\ne k}}$.
Thus the index $k$ can be eliminated. 
This reduction can be repeated to remove all indices that equal zero.

%
Let $W_j \ceq \phi_j(W)$.
By Lemma~\ref{lemma:factorization}, $s\in\PP\tup{W,\tup{W_j},\tup{\restr{\phi_j}{W}}}$.
Therefore by the inductive hypothesis, one of the inequalities in \eqref{BLfield} is
\begin{equation} \label{Winequality}
\sum_{x\in W} \prod_{j=1}^m f_j(\phi_j(x))\le \prod_{j=1}^m \norm{\restr{f_j}{W_j}}_{1/s_j}.
\end{equation}

%

Define $F_j \from V_j/W_j\to[0,\infty)$ to be the function 
\begin{equation*}
F_j(x+W_j) = \Big(\sum_{y\in W_j} f_j(y+x)^{1/s_j}\Big)^{s_j}.
\end{equation*}
This quantity is a function of the coset $x+W_j$ alone, rather than of $x$ itself, because for any $z\in W_j$,
\begin{equation*}
\sum_{y\in W_j} f_j(y+(x+z))^{1/s_j} =\sum_{y\in W_j} f_j(y+x)^{1/s_j}
\end{equation*}
by virtue of the substitution $y+z\mapsto y$.
Moreover,
\begin{equation} \label{Fjnorm} 
\norm{F_j}_{1/s_j} = \norm{f_j}_{1/s_j}.
\end{equation}
To prove this, choose one element $x\in V_j$ for each coset $x+W_j\in V_j/W_j$.
Denoting by $X$ the set of all these representatives, 
\begin{equation*}
\norm{F_j}_{1/s_j}^{1/s_j}  
=\sum_{x\in X} \sum_{y\in W_j} f_j(y+x)^{1/s_j} 
=\sum_{z\in V_j}  f_j(z)^{1/s_j}
\end{equation*}
because the map $X\times W_j\ni (x,y)\mapsto x+y\in V_j$ is a bijection.

The inductive bound \eqref{Winequality} can be equivalently written in the more general form
\begin{equation} \label{Winequality2}
\sum_{x\in W} \prod_{j=1}^m f_j(\phi_j(x+y))\le \prod_{j=1}^m F_j([\phi_j](y+W))
\end{equation}
for any $y\in V$, by applying \eqref{Winequality} to $\tup{\hat{f}_j}$ where $\hat{f}_j(z) = f_j(z + \phi_j(y))$.

Denote by $Y\subseteq V$ a set of representatives of the cosets $y+W\in V/W$, and identify $V/W$ with $Y$. 
Then
\begin{equation*} 
\sum_{x\in V} \prod_{j=1}^m f_j(\phi_j(x))
= \sum_{y\in Y} \sum_{x\in W} \prod_{j=1}^m f_j(\phi_j(y+x))
\le \sum_{y\in Y} \prod_{j=1}^m F_j([\phi_j](y+W)) 
\end{equation*}
by \eqref{Winequality2}.

By \eqref{Fjnorm}, it suffices to show that
\begin{equation} \label{quotientBL}
\sum_{y\in Y} \prod_{j} F_j([\phi_j](y+W))\le \prod_j\norm{F_j}_{1/s_j}  \qquad \text{for all functions } 0 \le F_j \in \ell^{1/s_j}(V_j/W_j).
\end{equation}
%
%
This is 
a set of inequalities of exactly the form \eqref{BLfield}, with $\fv$ replaced by 
$\fv_{V/W}=\tup{V/W,\tup{V_j/W_j},\tup{[\phi_j]}}$.
By Lemma~\ref{lemma:factorization}, $s\in\PP(\fv_{V/W})$, and since $\Dim{V/W}<\Dim{V}$, we conclude directly from the inductive hypothesis that \eqref{quotientBL} holds, concluding the proof of Lemma~\ref{lemma:criticalW}.
\end{proof}
%

According to Proposition~\ref{prop:0or1}, in order to complete the proof of Theorem~\ref{thm:field}, it remains 
only to analyze the case in which the extreme point $s$ is an element of $\set{0,1}^m$. 
Let $K=\set{k: s_k=1}$. 
Consider $W=\bigcap_{k\in K} \Kernel{\phi_k}$. 
Since $W$ is subcritical by hypothesis,
\begin{equation*}
\Dim{W}\le \sum_{j=1}^m s_j\Dim{\phi_j(W)} = \sum_{k\in K}\Dim{\phi_k(W)}=0 
\end{equation*}
so $\Dim{W}=0$, that is, $W=\set{0}$.
Therefore the map $x\mapsto\tup{\phi_k(x)}_{k\in K}$ from $V$ to the Cartesian product $\prod_{k\in K} V_k$ is injective. 

For any $x\in V$, 
\begin{equation*}
\prod_{j=1}^m f_j(\phi_j(x))
\le \prod_{k\in K} f_k(\phi_k(x))\prod_{i\notin K} \norm{f_i}_\infty
= \prod_{k\in K} f_k(\phi_k(x))\prod_{i\notin K} \norm{f_i}_{1/s_i}
\end{equation*}
since $s_i=0$ for all $i\notin K$.
Thus it suffices to prove that
\begin{equation*}
\sum_{x\in V}\prod_{k\in K} f_k(\phi_k(x)) \le \prod_{k\in K} \norm{f_k}_1.
\end{equation*}
This is a special case of the following result.
\begin{lemma} \label{lemma:exponentsallone}
Let $V$ be any finite-dimensional vector space over $\FF$.
Let $K$ be a finite index set, and for each $k\in K$, let $\phi_k$ be an $\FF$--linear map from $V$ to a finite-dimensional vector space $V_k$.
If $\bigcap_{k\in K}\Kernel{\phi_k}=\set{0}$ then for all functions $f_k\from V_k\to[0,\infty)$, 
\begin{equation*}
\sum_{x\in V} \prod_{k\in K} f_k(\phi_k(x))\le \prod_{k\in K}\norm{f_k}_1.
\end{equation*}
\end{lemma}
\begin{proof}
Define $\Phi\from V\to\prod_{k\in K} V_k$ by $\Phi(x) = \tup{\phi_k(x)}_{k\in K}$.
The hypothesis $\bigcap_{k\in K}\Kernel{\phi_k}=\set{0}$ is equivalent to $\Phi$ being injective.
The product $\prod_{k\in K} \norm{f_k}_1$ can be expanded as the sum of products
\begin{equation*}
\sum_{y} \prod_{k\in K} f_k(y_k)
\end{equation*}
where the sum is taken over all $y=\tup{y_k}_{k \in K}$ belonging to the Cartesian product $\prod_{k\in K}V_k$.
The quantity of interest, 
\begin{equation*}
\sum_{x\in V} \prod_{k\in K} f_k(\phi_k(x)),
\end{equation*}
is likewise a sum of such products.
Each term of the latter sum appears as a term of the former sum, and by virtue of the injectivity of $\Phi$, appears only once.
Since all summands are nonnegative, the former sum is greater than or equal to the latter. 
Therefore
\begin{equation*}
\prod_{k\in K}\norm{f_k}_1 
= \sum_{y} \prod_{k\in K} f_k(y_k) 
\ge \sum_{x\in V} \prod_{k\in K} f_k(\phi_k(x)).
\end{equation*}
\end{proof}

As mentioned above, necessity of the condition
\eqref{subcriticalhypothesisfield} for the inequality \eqref{BLfield} 
in the case $\FF=\QQ$ can be deduced from the corresponding necessity in Theorem~\ref{thm:2new}, by clearing denominators.
First, we identify $V$ and $V_j$ with $\QQ^d$ and $\QQ^{d_j}$ and let $E$ be any nonempty finite subset of $\QQ^d$.
Let $\hat \phi_j \from \QQ^d \to \QQ^{d_j}$ be the linear map represented by the matrix of $\phi_j$ multiplied by the lowest common denominator of its entries, i.e., an integer matrix.
Likewise, let $\hat E$ be the set obtained from $E$ by multiplying each point by the lowest common denominator of the coordinates of all points in $E$.
Then by linearity, \[ |\hat E| = |E| \le \prod_{j=1}^m |\phi_j(E)|^{s_j} = \prod_{j=1}^m |\hat \phi_j(\hat E)|^{s_j}. \]
Recognizing $\tup{\ZZ^d,\tup{\ZZ^{d_j}},\tup{\restr{\hat \phi_j}{\ZZ^d}}}$ as an Abelian group HBL datum, 
we conclude \eqref{subcriticalhypothesis2} for this datum from the implication \eqref{BL} $\Rightarrow$ \eqref{subcriticalhypothesis2} of Theorem~\ref{thm:2new}.
According to Lemma~\ref{lem:PZequalsPQ}, \eqref{subcriticalhypothesisfield} holds for the vector space HBL datum $\tup{\QQ^d,\tup{\QQ^{d_j}},\tup{\hat \phi_j}}$; our conclusion follows since $\Dim{\hat \phi_j (W)} = \Dim{\phi_j(W)}$ for any $W \subim \QQ^d$. 

It remains to show that
\eqref{subcriticalhypothesisfield} is necessary for \eqref{BLfield} 
in the case of a finite field $\FF$. 
Whereas the above reasoning required only the validity of \eqref{BLfield;sets} in the weakened form $|E|\le C\prod_{j=1}^m|\phi_j(E)|^{s_j}$ for some constant $C<\infty$ independent of $E$ (see proof of necessity for Theorem~\ref{thm:2new}), now the assumption that this holds with $C=1$ becomes essential.
Let $W$ be any subspace of $V$.
Since $|\FF|<\infty$ and $W$ has finite dimension over $\FF$, $W$ is a finite set and the hypothesis \eqref{BLfield;sets} can be applied with $E=W$.
Therefore $|W|\le\prod_{j=1}^m|\phi_j(W)|^{s_j}$.
This is equivalent to
\[ |\FF|^{\Dim{W}} \le \prod_{j=1}^m |\FF|^{s_j\Dim{\phi_j(W)}}, \]
so since $|\FF|\ge 2$, by taking base--$|\FF|$ logarithms of both sides, we obtain $\Dim{W} \le \sum_{j=1}^m s_j\Dim{\phi_j(W)}$, as was to be shown.
\qed
\section{An algorithm that computes the polytope $\PP$} \label{sec:computeP}
From the perspective of potential applications to communication bounds and the analysis of algorithms,
it is desirable to compute the convex  polytope $\PP=\PP(\fg)$  associated to an Abelian group
HBL datum.

%
We have already shown in Lemma~\ref{lem:PZequalsPQ} that $\PP$ is unchanged when 
the groups $\ZZ^d$ and $\ZZ^{d_j}$ are embedded in the natural way
into the vector spaces $\QQ^d$ and $\QQ^{d_j}$ over $\QQ$ and the homomorphisms $\phi_j$ are viewed as $\QQ$--linear maps.
Thus $\PP$ is identical to the polytope  defined by the inequalities
\begin{equation} \label{eq:DimV}
\Dim{V} \leq \sum_{j=1}^m s_j \Dim{\phi_j(V)} \qquad \text{for all subspaces } V \subim \QQ^d
\end{equation}
and $0\le s_j\le 1$ for all indices $j$. 
Indeed, \eqref{eq:DimV} is the hypothesis \eqref{subcriticalhypothesisfield} of Theorem~\ref{thm:field} in the case $\FF=\QQ$.

We will show how to compute $\PP$ in the case $\FF = \QQ$. 
Throughout the remainder of this section, $V$ and $V_j$ denote finite-dimensional vector spaces over $\QQ$, and $\phi_j$ denotes a $\QQ$--linear map.
The reasoning presented below applies to any countable field $\FF$, provided that elements of $\FF$ and the field operations are computable.

%
%
%
%

%
\begin{remark} \label{rmk:V07}
The algorithm described below relies on a search of a list of all subspaces of $V$.
A similar algorithm was sketched, less formally, in \cite{Valdimarsson07} for computing the corresponding polytope in \cite[Theorem~2.1]{BCCT10}. 
%
That algorithm searches a smaller collection of subspaces, namely the lattice generated by the kernels of $\phi_1,\ldots,\phi_m$ 
under the operations of intersection and pairwise sum of subspaces.
In a forthcoming sequel to this work, 
we will show that it suffices to search a corresponding lattice in our situation.
Searching this lattice would make the algorithm below more efficient, in principle. But our goal here is to
prove decidability, leaving issues of efficiency for future work.
A minor point is that this modification also allows relaxation of the requirement that $\FF$ be countable. 

\end{remark}

The proof of Theorem~\ref{thm:decision} is built upon several smaller results.

\begin{lemma} \label{lemma:enumeratesubspaces}
There exists an algorithm that takes as input a finite-dimensional vector space $V$ over 
$\QQ$, and returns a list of its subspaces.
More precisely, this algorithm takes as input a finite-dimensional vector space $V$ and a positive integer $N$, and returns as output the first $N$ elements $W_i$ of a list $\tup{W_1,W_2,\ldots}$ of all subspaces of $V$.
This list is independent of $N$.
Each subspace $W$ is expressed as a finite sequence $\tup{d;w_1,\ldots,w_d}$ where $d=\Dim{W}$ and $\set{w_i}$ is a basis for $W$.
\end{lemma}
\begin{proof}
Generate a list of all nonempty subsets of $V$ having at most $\Dim{V}$ elements.
Test each subset for linear independence, and discard all that fail to be independent.
Output a list of those that remain.
\end{proof}
We do not require this list to be free of redundancies.

\begin{lemma} \label{lemma:enumeratevertices}
For any positive integer $m$, there exists an algorithm that takes as input a finite set of linear inequalities over $\ZZ$ for $s\in[0,1]^m$, and returns as output a list of all the extreme points of the convex subset $\PP\subseteq [0,1]^m$ specified by these inequalities.
\end{lemma}
\begin{proof}
To the given family of inequalities, adjoin the $2m$ inequalities $s_j\ge 0$ and $-s_j\ge -1$. 
$\PP$ is the convex polytope defined by all inequalities in the resulting enlarged family. 
Express these inequalities as $\langle s,v_\alpha\rangle\ge c_\alpha$ for all $\alpha\in A$, where $A$ is a finite nonempty index set.

An arbitrary point $\tau\in\RR^m$ is an extreme point of $\PP$ if and only if 
(i) there exists a set $B$ of indices $\alpha$ having cardinality $m$, 
such that $(v_\beta: \beta\in B)$ is linearly independent and $\langle \tau,v_\beta\rangle = c_\beta$ for all $\beta\in B$, and 
(ii) $\tau$ satisfies $\langle \tau,v_\alpha\rangle\ge c_\alpha$ for all $\alpha\in A$.

Create a list of all subsets $B \subset A$ 
with cardinality equal to $m$.
There are finitely many such sets, since $A$ itself is finite.
Delete each one for which $(v_\beta: \beta\in B)$ is not linearly independent. 
For each subset $B$ not deleted, compute the unique solution $\tau$ of the system of equations $\langle \tau,v_\beta\rangle = c_\beta$ for all $\beta\in B$. 
Include $\tau$ in the list of all extreme points, if and only if $\tau$ satisfies $\langle \tau,v_\alpha\rangle \ge c_\alpha$ for all $\alpha\in A\setminus B$. 
\end{proof}

\begin{proposition} \label{prop:subalg}
There exists an algorithm that takes as input a vector space HBL datum $\fv=\tup{V,\tup{V_j},\tup{\phi_j}}$, an element $t\in[0,1]^m$, and a subspace $\set{0} \subpr W \subpr V$ which is critical with respect to $t$, and determines whether $t\in\PP(\fv)$.
\end{proposition}

Theorem~\ref{thm:decision} and Proposition~\ref{prop:subalg} will be proved inductively in tandem, according to the following induction scheme. 
The proof of Theorem~\ref{thm:decision} for HBL data in which $V$ has dimension $n$, will rely on Proposition~\ref{prop:subalg} for HBL data in which $V$ has dimension $n$.
The proof of Proposition~\ref{prop:subalg} for HBL data in which $V$ has dimension $n$ and there are $m$ subspaces $V_j$, will rely on Proposition~\ref{prop:subalg} for HBL data in which $V$ has dimension strictly less than $n$, on Theorem~\ref{thm:decision} for HBL data in which $V$ has dimension strictly less than $n$, and also on Theorem~\ref{thm:decision} for HBL data in which $V$ has dimension $n$ and the number of subspaces $V_j$ is strictly less than $m$.
Thus there is no circularity in the reasoning.

\begin{proof}[Proof of Proposition~\ref{prop:subalg}]
Let $\tup{V,\tup{V_j},\tup{\phi_j}}$ and $t,W$ be given. 
Following Notation~\ref{not:splittingspaces}, 
consider the two HBL data $\fv_W=\tup{W,\tup{\phi_j(W)},\tup{\restr{\phi_j}{W}}}$ 
and $\fv_{V/W}=\tup{V/W,\tup{V_j/\phi_j(W)},\tup{[\phi_j]}}$, where $[\phi_j] \from V/W\to V_j/\phi_j(W)$ are the quotient maps.
From a basis for $V$, bases for $V_j$, a basis for $W$, and corresponding matrix representations of $\phi_j$, it is possible to compute the dimensions of, and bases for, $V/W$ and  $V_j/\phi_j(W)$, via row operations on matrices.
According to Lemma~\ref{lemma:factorization}, $t\in\PP(\fv)$ if and only if 
$t\in\PP(\fv_W)\cap\PP(\fv_{V/W})$.

Because $\{0\} \subpr W \subpr V$, both $W,V/W$ have dimensions strictly less than the dimension of $V$. 
Therefore by Theorem~\ref{thm:decision} and the induction scheme, there exists an algorithm which computes both a finite list of inequalities characterizing $\PP(\fv_W)$, and a finite list of inequalities characterizing 
$\PP(\fv_{V/W})$.
Testing each of these inequalities on $t$ determines whether $t$ belongs to these two polytopes, hence whether $t$ belongs to $\PP(\fv)$.
\end{proof}

\begin{lemma} \label{lemma:cleanup} 
Let an HBL datum $\fv=\tup{V,\tup{V_j},\tup{\phi_j}}$ be given. 
Let $i\in\set{1,2,\ldots,m}$.
Let $s\in[0,1]^m$ and suppose that $s_i=1$.
Let $V'=\Kernel{\phi_i}$.
Define $\widehat{s}\in[0,1]^{m-1}$ to be $(s_1,\ldots,s_m)$ with the $i^\text{th}$ coordinate deleted.
Then $s\in\PP\tup{V,\tup{V_j},\tup{\phi_j}}$ if and only if $\widehat{s}\in\PP\tup{V',\tup{V_j}_{j\ne i},\tup{\restr{\phi_j}{V'}}_{j\ne i}}$.
\end{lemma}
\begin{proof}
Write $\fv' = \tup{V',\tup{V_j}_{j\ne i},\tup{\restr{\phi_j}{V'}}_{j\ne i}}$.
For any subspace $W \subim V'$, since $\Dim{\phi_i(W)}=0$,
\[ \sum_j s_j \Dim{\phi_j(W)} = \sum_{j\ne i} s_j \Dim{\phi_j(W)}. \]
So if $s\in\PP(\fv)$ then $\widehat{s}\in\PP(\fv')$.

Conversely, suppose that $\widehat{s}\in\PP(\fv')$.
Let $W$ be any subspace of $V$. 
Write $W=W''+(W\cap V')$ where the subspace $W'' \subim V$ is a supplement to $W \cap V'$ in $W$, so that $\Dim{W} = \Dim{W''} + \Dim{W\cap V'}$.
Then
\begin{align*}
\sum_j s_j \Dim{\phi_j(W)} &= \Dim{\phi_i(W)} + \sum_{j\ne i}s_j \Dim{\phi_j(W)}
\\ &\ge  \Dim{\phi_i(W'')} + \sum_{j\ne i}s_j \Dim{\phi_j(W\cap V')}
\\ &\ge  \Dim{W''} + \Dim{W\cap V'};
\end{align*}
$\Dim{\phi_i(W'')} = \Dim{W''}$ because $\phi_i$ is injective on $W''$.
So $s\in\PP(\fv)$.
\end{proof}

To prepare for the proof of Theorem~\ref{thm:decision}, let $\PP(\fv)=\PP\tup{V,\tup{V_j},\tup{\phi_j}}$ be given.
Let $\tup{W_1,W_2,W_3,\ldots}$ be the list of subspaces of $V$ produced by the algorithm of Lemma~\ref{lemma:enumeratesubspaces}. 
Let $N\ge 1$.
To each index $\alpha \in\set{1,2,\ldots,N}$ is associated a linear inequality $\sum_{j=1}^m s_j\Dim{\phi_j(W_\alpha)} \ge \Dim{W_\alpha}$ for elements $s\in[0,1]^m$, which we encode by an $(m+1)$--tuple $\tup{v(W_\alpha),c(W_\alpha)}$; the inequality is $\langle s,v(W_\alpha)\rangle\ge c(W_\alpha)$.
Define $\PP_N\subseteq[0,1]^m$ to be the polytope defined by this set of inequalities.

\begin{lemma} \label{lemma:musthalt}
\begin{equation} \label{tautologicalinclusion}
\PP_N \supseteq \PP(\fv) \qquad \text{for all $N$}. 
\end{equation}
Moreover, there exists a positive integer $N$ such that $\PP_M=\PP(\fv)$ for all $M\ge N$.
\end{lemma}
\begin{proof}
The inclusion holds for every $N$, because the set of inequalities defining $\PP_N$ is a subset of the set 
defining $\PP(\fv)$.

$\PP(\fv)$ is specified by some finite set of inequalities, each specified by some subspace of $V$.
Choose one such subspace for each of these inequalities.
Since $\tup{W_\alpha}$ is a list of all subspaces of $V$, there exists $M$ such that each of these chosen subspaces belongs to $\tup{W_\alpha: \alpha\le M}$. 
\end{proof}

\begin{lemma}
Let $m\ge 2$.
If $s$ is an extreme point of $\PP_N$, then either $s_j\in\set{0,1}$ for some $j\in\set{1,2,\ldots,m}$, or there exists $\alpha\in\set{1,2,\ldots,N}$ for which $W_\alpha$ is critical with respect to $s$ and $0<\Dim{W_\alpha}<\Dim{V}$.
\end{lemma}
In the following argument, we say that two inequalities $\langle s,v(W_\alpha)\rangle\ge c(W_\alpha)$, $\langle s,v(W_\beta)\rangle\ge c(W_\beta)$ are distinct if they specify different subsets of $\RR^m$.
\begin{proof}
For any extreme point $s$, equality must hold in at least $m$ distinct inequalities among those defining $\PP_N$.
These inequalities are of three kinds: $\langle s,v(W_\alpha)\rangle\ge c(W_\alpha)$ for $\alpha\in\set{1,2,\ldots,N}$, $s_j\ge 0$, and $-s_j\ge -1$, with $j\in\set{1,2,\ldots,m}$.
If $W_\beta=\set{0}$ then $W_\beta$ specifies the tautologous inequality $\sum_j s_j\cdot 0=0$, so that index $\beta$ can be disregarded.

If none of the coordinates $s_j$ are equal to $0$ or $1$, there must exist $\beta$ such that equality holds in at least two distinct inequalities $\langle s,v(W_\beta)\rangle \ge  c(W_\beta)$ associated to subspaces $W_\alpha$ among those which are used to define $\PP_N$.
We have already discarded the subspace $\set{0}$, so there must exist $\beta$ such that $W_\beta$ and $V$ specify distinct inequalities.
Thus $0<\Dim{W_\beta}<\Dim{V}$.
\end{proof}

\begin{proof}[Proof of Theorem~\ref{thm:decision}]
Set $\PP=\PP\tup{V,\tup{V_j},\tup{\phi_j}}$.
Consider first the base case $m=1$. 
The datum is a pair of finite-dimensional vector spaces $V,V_1$ with 
a $\QQ$--linear map $\phi\from V\to V_1$.
The polytope $\PP$ is the set of all $s\in[0,1]$ for which $s\Dim{\phi(W)}\ge \Dim{W}$ for every subspace $W\subim V$. 
If $\Dim{V}=0$ then $\PP\tup{V,\tup{V_j},\tup{\phi_j}} = [0,1]$.
If $\Dim{V}>0$ then since $\Dim{\phi(W)} \le \Dim{W}$ for every subspace, the inequality can only hold if the kernel
of $\phi$ has dimension $0$, and then only for $s=1$.
The kernel of $\phi$ can be computed.
So  $\PP$ can be computed when $m=1$.

Suppose that $m\ge 2$.
Let $\tup{V,\tup{V_j},\tup{\phi_j}}$ be given.
Let $N=0$. 
Recursively apply the following procedure.

Replace $N$ by $N+1$. 
Consider $\PP_N$. 
Apply Lemma~\ref{lemma:enumeratevertices} to obtain a list of all extreme points $\tau$ of $\PP_N$, and for each such $\tau$ which belongs to $(0,1)^m$, a nonzero proper subspace $W(\tau)\subim V$ which is critical with respect to $\tau$. 

Examine each of these extreme points $\tau$, to determine whether $\tau\in\PP\tup{V,\tup{V_j},\tup{\phi_j}}$.
There are three cases. 
Firstly, if $\tau\in (0,1)^m$, then Proposition~\ref{prop:subalg} may be invoked, using the critical subspace $W(\tau)$, to determine whether $\tau\in \PP$.

Secondly, if some component $\tau_i$ of $\tau$ equals $1$, let $V'$ be the kernel of $\phi_i$.
Set 
\[
\PP'=\PP\dtup{V',\dtup{V_j}_{j\ne i},\dtup{\restr{\phi_j}{V'}}_{j\ne i}}.
\]
According to Lemma~\ref{lemma:cleanup}, $\tau\in\PP$ if and only if $\widehat{\tau} = \tup{\tau_j}_{j\ne i}\in\PP'$. 
This polytope $\PP'$ can be computed by the induction hypothesis, since the number of indices $j$ has been reduced by one.

Finally, if some component $\tau_i$ of $\tau$ equals $0$, then because the term $s_i\Dim{\phi_i(W)}=0$ contributes nothing to sums $\sum_{j=1}^m s_j \Dim{\phi_j(W)}$, $\tau\in \PP$ if and only if $\widehat{\tau}$ belongs to $\PP\tup{V,\tup{V_j}_{j\ne i},\tup{\phi_j}_{j\ne i}}$.
To determine whether $\widehat{\tau}$ belongs to this polytope requires again only an application of the induction hypothesis.

If every extreme point $\tau$ of $\PP_N$ belongs to $\PP$, then because $\PP_N$ is the convex hull of its extreme points, $\PP_N\subseteq \PP$.
The converse inclusion holds for every $N$, so in this case $\PP_N= \PP$.
The algorithm halts, and returns the conclusion that $\PP =\PP_N$, along with information already computed: a list of the inequalities specified by all the subspaces  $W_1,\ldots,W_N$, and a list of extreme points of $\PP_N = \PP$.

On the other hand, if at least one extreme point of $\PP_N$ fails to belong to $\PP$, then $\PP_N\ne\PP$.
Then increment $N$ by one, and repeat the above steps. 

Lemma~\ref{lemma:musthalt} guarantees that this procedure will halt after finitely many steps. 
\end{proof}

\section{On (un)computability of the list of constraints defining $\PP$}
\label{sec:undecidability}


%
We have discussed the computation of $\PP$. Now we turn to the problem of computing the set of inequalities \eqref{subcriticalhypothesis2}.
There may be many sets of 
inequalities which specify $\PP$; thus in order to compute $\PP$, it suffices to compute any such set of inequalities, rather than 
the specific set \eqref{subcriticalhypothesis2}.
This distinction accounts for the coexistence of Theorems~\ref{thm:decision} and \ref{thm:hilbert}. 


In order to compute the set of inequalities \eqref{subcriticalhypothesis2}, we must compute the answer to the question:
Given any group homomorphisms $\phi_1, \ldots, \phi_m$ and integers $0 \leq r,r_1,\ldots, r_m \leq d$, does there exist a subgroup $H \subim \ZZ^d$ such that
\begin{equation*}
\Rank{H} = r, 
\ \text{ and } \ \Rank{\phi_j(H)}=r_j
\ \text{ for all }\ 1\le j\le m?
\end{equation*}
%
%

%
\begin{notation}
For a natural number $d$ and ring $R$, $\Mat{d}{R}$ denotes the ring of $d$--by--$d$ matrices with entries from $R$.
%
%
 We identify $\Mat{d}{R}$ with the endomorphism ring of the $R$--module $R^d$
(or $R$-vector space, if $R$ is a field) 
and thus may write elements of $\Mat{d}{R}$ as $R$--linear maps rather than as matrices. 
 Via the usual coordinates, we may identify $\Mat{d}{R}$ with $R^{d^2}$. 
 We write $R[x_1,\ldots,x_q]$ to denote the ring of polynomials over $R$ in variables $x_1,\ldots,x_q$.
\end{notation}

Recall that there are given $d,d_j\in\NN$ and $\ZZ$--linear maps $\phi_j\from\ZZ^d\to\ZZ^{d_j}$, for $j\in\set{1,2,\ldots,m}$ for some positive integer $m$.
%
%
%
%
Each $\phi_j$ can also be interpreted as a $\QQ$--linear map (from $\QQ^d$ to $\QQ^{d_j}$), represented by the same integer matrix.
It is no loss of generality to assume that each $d_j=d$, so that each $\phi_j$ is an endomorphism of $\ZZ^d$. 

\begin{definition} 
\label{defZtoQ}
Given $m,d \in \NN$, and a finite sequence $r,r_1,\ldots,r_m$ of natural numbers each bounded by $d$, we define the sets
\begin{align*}
E_{d;r,r_1,\ldots,r_m} &\ceq \dset{(\phi_1,\ldots,\phi_m) \in (\Mat{d}{\ZZ})^m : (\exists H \subim \ZZ^d)\; \Rank{H} = r \text{ and } \Rank {\phi_j(H)} = r_j, 1 \le j \le m }, \\
E_{d;r,r_1,\ldots,r_m}^\QQ &\ceq \dset{(\phi_1,\ldots,\phi_m) \in (\Mat{d}{\QQ})^m : (\exists V \subim \QQ^d)\; \Dim{V} = r \text{ and } \Dim{\phi_j(V)} = r_j, 1 \le j \le m }.
\end{align*}
%
%
\end{definition}

\begin{remark} \label{rmk:minorpolynomials}
The question of whether a given $m$--tuple $(\phi_1,\ldots,\phi_m) \in (\Mat{d}{R})^m$ is a member of $E_{d;r,r_1,\ldots,r_m}$ (when $R=\ZZ$) or $E_{d;r,r_1,\ldots,r_m}^\QQ$ (when $R=\QQ$) is an instance of the problem of whether some system of polynomial equations has a solution over the ring $R$.
We let $B$ be a $d$--by--$r$ matrix of variables, and construct a system of polynomial equations in the $dr$ unknown entries of $B$ and $md^2$ known entries of $\phi_1,\ldots,\phi_m$ that has a solution if and only if the aforementioned rank (or dimension) conditions are met.
The condition $\Rank{M} = s$ for a matrix $M$ is equivalent to all $(s+1)$--by--$(s+1)$ minors of $M$ equaling zero (i.e., the sum of their squares equaling zero), and at least one $s$--by--$s$ minor being nonzero (i.e., the sum of their squares not equaling zero --- see Remark~\ref{rmk:5.1}). 
We construct two polynomial equations in this manner for $M=B$ (with $s=r$) and for each matrix $M=\phi_j B$ (with $s=r_j$).
\end{remark}

\begin{lemma} 
\label{rationalgroup}
With the notation as in Definition~\ref{defZtoQ}, $E_{d;r,r_1,\ldots,r_m} = E_{d;r,r_1,\ldots,r_m}^\QQ \cap (\Mat{d}{\ZZ})^m$.
\end{lemma}
\begin{proof}
This result was already established in Lemma~\ref{lem:PZequalsPQ}; we restate it here using the present notation.
For the left-to-right inclusion, observe that if $H \subim \ZZ^d$ witnesses that $(\phi_1,\ldots,\phi_m) \in E_{d;s,r_1,\ldots,r_m}$, then $H_\QQ$ witnesses that $(\phi_1,\ldots,\phi_m) \in E_{d;r,r_1,\ldots,r_m}^\QQ$.
For the other inclusion, if  $(\phi_1,\ldots,\phi_m) \in E_{d;r,r_1,\ldots,r_m}^\QQ \cap (\Mat{d}{\ZZ})^m$ witnessed by $V \subim \QQ^d$, then we may find a submodule $H = V\cap\ZZ^d$ of $\ZZ^d$, with $\Rank{H}=\Dim{V}$.
Then $\Dim{\phi_j(V)}  = \Rank{\phi_j(H)} = r_j$ showing that $(\phi_1,\ldots,\phi_m) \in E_{d;r,r_1,\ldots,r_m}$. 
\end{proof}

\begin{definition} \label{def:HTPQ}
\emph{Hilbert's Tenth Problem for $\QQ$} is the question of whether there is an algorithm which given a finite set of polynomials $f_1(x_1,\ldots,x_q),\ldots, f_p(x_1,\ldots,x_q) \in \QQ[x_1,\ldots,x_q]$ (correctly) determines whether or not there is some $a \in \QQ^q$ for which $f_1(a) = \cdots = f_p(a) = 0$.
\end{definition}

\begin{remark} \label{rmk:5.1}
One may modify the presentation of Hilbert's Tenth Problem for $\QQ$ in various ways without affecting its truth value.  
For example, one may allow a condition of the form $g(a) \neq 0$ as this is equivalent to $(\exists b) (g(a)b - 1 = 0)$.  
On the other hand, using the fact that $x^2 + y^2 = 0 \Longleftrightarrow x = 0 = y$, one may replace the finite sequence of polynomial equations with a single equality.
\end{remark}

\begin{remark}
Hilbert's Tenth Problem, proper, asks for an algorithm to determine solvability in integers of finite systems of equations over $\ZZ$.
 From such an algorithm one could positively resolve Hilbert's Tenth Problem for $\QQ$.  
 However, by the celebrated theorem of Matiyasevich-Davis-Putnam-Robinson~\cite{Matiyasevich93}, no such algorithm exists.   
 The problem for the rationals remains open.  
 The most natural approach would be to reduce from the problem over $\QQ$ to the problem over $\ZZ$, say, by showing that $\ZZ$ may be defined by an expression of the
form
$$
a \in \ZZ \Longleftrightarrow (\exists y_1) \cdots (\exists y_q) P(a;y_1,\ldots,y_q) = 0
$$
 for some fixed polynomial $P$.
 Koenigsmann~\cite{Kon} has shown that there is in fact a \emph{universal} definition of $\ZZ$ in $\QQ$, that is, a formula of the form
$$
a \in \ZZ \Longleftrightarrow (\forall y_1) \cdots (\forall y_q) \theta(a;y_1,\ldots,y_q) = 0
$$
 where $\theta$ is a finite Boolean combination of polynomial equations, but he also demonstrated that the existence of an existential
definition of $\ZZ$ would violate the Bombieri-Lang conjecture. 
Koenigsmann's result shows that it is unlikely that Hilbert's Tenth Problem for $\QQ$ can be resolved by reducing to the problem over $\ZZ$ using an existential definition of $\ZZ$ in $\QQ$. 
However, it is conceivable that this problem could be resolved without such a definition.
\end{remark}

\begin{proof}[Proof of Theorem~\ref{thm:hilbert} (necessity)]
Evidently, if Hilbert's Tenth Problem for $\QQ$ has a positive solution, then there is an algorithm to (correctly) determine for $d \in \NN$, $r, r_1, \ldots, r_m \leq d$ also in $\NN$, and $(\phi_1,\ldots,\phi_m) \in (\Mat{d}{\ZZ})^m$ whether $(\phi_1, \ldots, \phi_m) \in  E_{d;r,r_1,\ldots,r_m}$.
%
By Lemma~\ref{rationalgroup}, $(\phi_1,\ldots,\phi_m) \in E_{d;r,r_1,\ldots,r_m}$ just in case $(\phi_1,\ldots,\phi_m) \in E_{d;r,r_1,\ldots,r_m}^\QQ$.  
This last condition leads to an instance of Hilbert's Tenth Problem (for $\QQ$) for the set of rational polynomial equations given in Remark~\ref{rmk:minorpolynomials}.
\end{proof}

\begin{notation}
Given a set $S \subseteq \QQ[x_1,\ldots,x_q]$ of polynomials, we denote the set of rational solutions to the equations $f = 0$ as $f$ ranges through $S$ by
$$
V(S)(\QQ) \ceq \dset{ a \in \QQ^q : (\forall f \in S) f(a) = 0 }.
$$
\end{notation}
%

\begin{definition}
For any natural number $t$, we say that the set $D \subseteq \QQ^t$ is \emph{Diophantine} if there is some $q-t \in \NN$ and a set 
$S \subseteq \QQ[x_1,\ldots,x_t;y_1,\ldots,y_{q-t}]$ for which 
$$
D = \dset{ a \in \QQ^t : (\exists b \in \QQ^{q-t})  (a;b) \in V(S)(\QQ) }.
$$
\end{definition}
\noindent We will show sufficiency in Theorem~\ref{thm:hilbert} by establishing a stronger result: namely, that an algorithm to decide membership in sets of the form $E_{d;r,r_1,\ldots,r_m}$ could also be used to decide membership in any Diophantine set. 
(Hilbert's Tenth Problem for $\QQ$ concerns membership in the specific Diophantine set $V(S)(\QQ)$.)

With the next lemma, we use a standard trick of replacing composite terms with single applications of the basic operations to put a general Diophantine set in a standard form (see, e.g., \cite{Vak}).
\begin{lemma}\label{lem:basicset}
Given any finite set of polynomials $S \subset \QQ[x_1,\ldots,x_q]$, let $d \ceq \max_{f\in S} \max_{i=1}^q \deg_{x_i}(f)$ and $\mathcal{D} \ceq \set{0,1,\ldots,d}^q$.
There is another set of variables $\set{u_\alpha}_{\alpha \in \mathcal{D}}$ and another finite set $S' \subset \QQ [ \set{u_\alpha}_{\alpha \in \mathcal{D}} ]$ consisting entirely of affine polynomials (polynomials of the form $c + \sum c_\alpha u_\alpha$ where not all $c$ and $c_\alpha$ are zero) and polynomials of the form $u_\alpha u_\beta - u_\gamma$ with $\alpha$, $\beta$, and $\gamma$ distinct, so that $V(S)(\QQ) = \pi (V(S')(\QQ))$ where $\pi\from\QQ^{\mathcal{D}} \to \QQ^q$ is given by 
$$
\dtup{u_\alpha}_{\alpha \in \mathcal{D}} \mapsto \dtup{u_{(1,0,\ldots,0)},u_{(0,1,0,\ldots,0)}, \ldots, u_{(0,\ldots,0,1)}}.
$$ 
\end{lemma} 
\begin{proof}
Denote by ${\mathbf 0}$ the element $(0,0,\dots,0)$ of $\mathcal{D}=\{0,1,\dots,d\}^q$.
Let $T \subset \QQ [ \set{ u_\alpha}_{\alpha \in \mathcal{D}} ]$ consist of
\begin{itemize}
\item $u_{(0,\ldots,0)} - 1$ and 
\item $u_{\alpha + \beta} - u_\alpha u_\beta$ for $(\alpha + \beta) \in \mathcal{D}$, $\alpha \neq {\mathbf 0}$ and $\beta \neq {\mathbf 0}$. 
\end{itemize}
Define $\chi\from\QQ^q \to \QQ^{\mathcal{D}}$ by 
$$
\dtup{x_1,\ldots,x_q} \mapsto \dtup{x^\alpha}_{\alpha \in \mathcal{D}}
$$
where $x^\alpha \ceq \prod_{j=1}^q x_j^{\alpha_j}$.
One sees immediately that $\chi$ induces a bijection $\chi\from\QQ^q \to V(T)(\QQ)$ with inverse $\restr{\pi}{V(T)(\QQ)}$.

Let $S'$ be the set containing $T$ and the polynomials $\sum_{\alpha \in \mathcal{D}} c_\alpha u_\alpha$ for which $\sum_{\alpha \in \mathcal{D}} c_\alpha x^\alpha \in S$.   

One checks that if $a \in \QQ^q$, then $\chi(a) \in V(S')(\QQ)$ if and only if $a \in V(S)(\QQ)$.  
Applying $\pi$, and noting that $\pi$ is the inverse to $\chi$ on $V(T)(\QQ)$, the result follows.
\end{proof}
\begin{notation}
For the remainder of this argument, we call a set enjoying the properties identified for $S'$ (namely that each polynomial is either affine or of the form $u_\alpha u_\beta - u_\gamma$) a \emph{basic set}.
\end{notation}


\begin{proof}[Proof of Theorem~\ref{thm:hilbert} (sufficiency)]
It follows from Lemma~\ref{lem:basicset} that the membership problem for a general 
finite
Diophantine set may be reduced to the membership problem for a Diophantine set defined by a finite basic set of equations.
Let  $S \subseteq \QQ[x_1,\ldots,x_q]$ be a finite basic set of equations and let $t \leq q$ be some natural number, 
We now show that there are natural numbers $\mu,\nu,\rho,\rho_1, \ldots, \rho_\mu$ and a computable function $f\from\QQ^t \to (\Mat{\nu}{\ZZ})^\mu$ so that for $a \in \QQ^t$ one has that there is some $b \in \QQ^{q-t}$ with $(a,b) \in V(S)(\QQ)$ if and only if $f(a) \in E_{\nu;\rho,\rho_1,\ldots,\rho_\mu}$.

List the $\ell$ affine polynomials in $S$ as
$$
\lambda_{0,1} + \sum_{i=1}^q \lambda_{i,1} x_i, \qquad \ldots,\qquad \lambda_{0,\ell} + \sum_{i=1}^q \lambda_{i,\ell} x_i
$$
and the $k$ polynomials expressing multiplicative relations in $S$ as
$$
x_{i_{1,1}} x_{i_{2,1}} - x_{i_{3,1}},\qquad \ldots,\qquad x_{i_{1,k}} x_{i_{2,k}} - x_{i_{3,k}}.
$$
Note that by scaling, we may assume that all of the coefficients $\lambda_{i,j}$ are integers. 
  
We shall take $\mu \ceq 4 + q + t + |S|$, $\nu \ceq 2q+2$, $\rho \ceq 2$ and the sequence $\rho_1, \ldots, \rho_\mu$ to consist of $4 + q + k$ ones followed by $t + \ell$ zeros. 
 Let us describe the map $f\from\QQ^t \to (\Mat{\nu}{\ZZ})^\mu$ by expressing each coordinate.  
 For the sake of notation, our coordinates on $\QQ^\nu$ are $(u;v) \ceq (u_0,u_1\ldots, u_q; v_0, v_1, \ldots, v_q)$.
\begin{itemize}
\item[A.] The map $f_1$ is constant taking the value $(u;v) \mapsto (u_0,0,\ldots,0)$.
\item[B.] The map $f_2$ is constant taking the value $(u;v) \mapsto (v_0,0,\ldots,0)$.
\item[C.] The map $f_3$ is constant taking the value $(u;v) \mapsto (u_0,\ldots,u_q;0,\ldots,0)$.
\item[D.] The map $f_4$ is constant taking the value $(u;v) \mapsto (v_0,\ldots,v_q;0,\ldots,0)$.
\item[E.] The map $f_{4+j}$ (for $0 < j \leq q$) is constant taking the value $(u;v) \mapsto (u_0 - v_0,u_j - v_j,0,\ldots,0)$.
\item[F.] The map $f_{4+q+j}$ (for $0 < j \leq k$) is constant taking the value $(u;v) \mapsto (u_{i_{1,j}} + v_0,u_{i_{3,j}} + v_{i_{2,j}},0,\ldots,0)$.
\item[G.] The map $f_{4 + q + k + j}$ (for $0 < j \leq t$) takes $a = (\frac{p_1}{q_1}, \ldots, \frac{p_t}{q_t})$ (written
in lowest terms) to the linear map $(u;v) \mapsto (p_j u_0 -  q_j u_j,0,\ldots,0)$.
\item[H.] The map $f_{4 + q + k + t + j}$ (for $0 < j \leq \ell$) takes the value $(u;v) \mapsto (\sum_{i=0}^q \lambda_{i,j} u_i,0,\ldots,0)$.
\end{itemize}
Note that only the components $f_{4 + q + k + j}$ for $0 < j \leq t$ actually depend on $(a_1,\ldots,a_t) \in \QQ^t$.

Let us check that this construction works.  
First, suppose that $a \in \QQ^t$ and that there is some $b \in \QQ^{q-t}$  for which $(a,b) \in V(S)(\QQ)$. 
For the sake of notation, we write $c = (c_1,\ldots,c_q) \ceq (a_1,\ldots,a_t,b_1,\ldots,b_{q-t})$.  

Let $V \ceq \QQ (1,c_1,\ldots,c_q,0,\ldots,0) + \QQ (0,\ldots,0,1,c_1,\ldots,c_q)$.  
We will check now that $V$ witnesses that $f(a) \in E_{\nu,\rho,\rho_1,\ldots,\rho_\mu}$.   
Note that a general element of $V$ takes the form $(\alpha, \alpha c_1, \ldots, \alpha c_q,\beta,\beta c_1,\ldots,\beta c_q)$ for $(\alpha,\beta) \in \QQ^2$.  
Throughout the rest of this proof, when we speak of a general element of some image of $V$, we shall write $\alpha$ and $\beta$ as variables over $\QQ$.

Visibly, $\Dim{V}=2=\rho$.

Clearly, $f_1(a)(V) = \QQ (1,0,\ldots,0) = f_2(a)(V)$, so that $\rho_1 = \Dim{f_1(a)(V)} = 1 = \Dim{f_2(a)(V)} = \rho_2$ as required.

Likewise, $f_3(a)(V) = \QQ(1,c_1,\ldots,c_q,0,\ldots,0) = f_4(a)(V)$, so that $\rho_3 = \Dim{f_3(a)(V)} = 1 = \Dim{f_4(a)(V)} = \rho_4$.

For $0 < j \leq q$ the general element of $f_{4+j}(a)(V)$ has the form $(\alpha - \beta, \alpha c_j - \beta c_j,0,\ldots,0) = (\alpha - \beta) (1,c_j,0,\ldots,0)$. Thus, $f_{4+j}(a)(V) = \QQ (1,c_j,0,\ldots,0)$ has dimension $\rho_{4+j}=1$.
 
For $0 < j \leq k$ we have  $c_{i_{3,j}} = c_{i_{1,j}} c_{i_{2,j}}$, the general element of $f_{4+q+j}(a)(V)$ has the form $(\alpha c_{i_{1,j}} + \beta,\alpha c_{i_{3,j}} + \beta c_{i_{2,j}},0,\ldots,0)  = (\alpha c_{i_{1,j}} + \beta, \alpha c_{i_{1,j}} c_{i_{2,j}} + \beta c_{i_{2,j}},0,\ldots,0) = (\alpha c_{i_{1,j}} + \beta) (1,c_{i_{2,j}},0,\ldots,0)$ so we have that $\rho_{4+q+j}=\Dim{f_{4+q+j}(a)(V)} = 1$.
 
 For $0 < j \leq t$, the general element of $f_{4 + q + k + j}(a)(V)$ has the form $(p_j \alpha - q_j \alpha a_j,0,\ldots,0) = {\mathbf 0}$.  
 That is, $\rho_{4+q+k+j} = \Dim{f_{4+q+k+j}(a)(V)} = 0$.
 
 Finally, if $0 < j \leq \ell$, then the general element of $f_{4+q+k+t+j}(a)(V)$ has the form $(\lambda_{0,j} \alpha + \sum_{i=1}^q \lambda_{i,j} \alpha c_i, 0, \ldots, 0) = {\mathbf 0}$ since $\lambda_{0,j} + \sum_{i=1}^q \lambda_{i,j} c_i = 0$. 
 So $\rho_{4+q+k+t+j} = \Dim{f_{4+q+k+t+j}(a)(V)}=0$.
 
 Thus, we have verified that if $a \in \QQ^t$ and there is some $b \in \QQ^{d-t}$ with $(a,b) \in V(S)(\QQ)$, then $f(a) \in E_{\nu;\rho,\rho_1,\ldots,\rho_\mu}$.
 
 Conversely, suppose that $a = (\frac{p_1}{q_1}, \ldots, \frac{p_t}{q_t}) \in \QQ^t$ and that $f(a) \in E_{\nu;\rho,\rho_1,\ldots,\rho_\mu}$, with the same $\mu,\nu,\rho,\rho_1,\ldots,\rho_\mu$. 
 Let  $V \subim \QQ^\nu$ have $\Dim{V} = \rho$ witnessing that $f(a) \in E_{\nu;\rho,\rho_1,\ldots,\rho_\mu}$.
 
 \begin{lemma} \label{claim1}
There are elements $g$ and $h$ in $V$ for which 
$g = (0,\ldots,0;g_0,\ldots,g_q)$, 
$h = (h_0,\ldots,h_q;0,\ldots,0)$, 
$g_0 = h_0 = 1$, and $\QQ g + \QQ h = V$.
\end{lemma}
\begin{proof}
 The following is an implementation of row reduction.
  Let the elements $d = (d_0,\ldots,d_q;d_0',\ldots,d_q')$ and $e = (e_0,\ldots,e_q;e_0',\ldots,e_q')$ be a basis for $V$.
 Since $\Dim{f_1(a)(V)} = 1$, at the cost of reversing $d$ and $e$ and multiplying by a scalar, we may assume that $d_0 = 1$.  
 Since $\Dim{f_3(a)(V)} = 1$, we may find a scalar $\gamma$ for which $(\gamma d_0, \ldots, \gamma d_q) = (e_0, \ldots, e_q)$.  
 Set $\widetilde{g} \ceq e - \gamma d$.    
 Write $\widetilde{g} = (0,\ldots,0,\widetilde{g}_0,\ldots,\widetilde{g}_q)$.
 Since $\widetilde{g}$ is linearly independent from $e$ and $\Dim{f_4(a)(V)} = 1$, we see that there is some scalar $\delta$ for which $(\delta \widetilde{g}_0, \ldots, \delta \widetilde{g}_q) = (d_0' , \ldots, d_q')$.  
 Set $h \ceq  d - \delta \widetilde{g}$.   
 Using the fact that $\Dim{f_2(a)(V)} = 1$ we see that $\widetilde{g}_0 \neq 0$.  
 Set $g \ceq \widetilde{g}_0^{-1} \widetilde{g}$. 
 \end{proof}
\begin{lemma}
 For $0 \leq j \leq q$ we have $g_j = h_j$.
\end{lemma}
\begin{proof}
We arranged $g_0 = 1 = h_0$ in Lemma~\ref{claim1}.
The general element of $V$ has the form $\alpha h + \beta g$ for some $(\alpha,\beta) \in \QQ^2$.
For $0 < j \leq q$, the general element of $f_{4+j}(a)(V)$ has the form $(\alpha h_0 - \beta h_0,\alpha h_j - \beta g_j, 0, \ldots, 0) = ( (\alpha - \beta) h_0, (\alpha - \beta) h_j + \beta (h_j - g_j), 0, \ldots, 0)$. 
Since $h_0 \neq 0$, if $h_j \neq g_j$, then this vector space would have dimension two, contrary to the requirement that $f(a) \in E_{\nu;\rho,\rho_1,\ldots,\rho_\mu}$. 
\end{proof}
\begin{lemma}
For $0 < j \leq t$ we have $h_j = a_j$.
\end{lemma}
\begin{proof}
The image of $\alpha h + \beta g$ under $f_{4 + q + k + j}(a)$ is $(p_j \alpha  - q_j \alpha h_j,0,\ldots,0)$ where $a_j = \frac{p_j}{q_j}$ in lowest terms.  
Since $\Dim{f_{4+q+k+j}(a)(V)} = 0$, we have $q_j h_j = p_j$.  
That is, $h_j = a_j$. 
\end{proof}
\begin{lemma}
For any $F \in S$, we have $F(h_1,\ldots,h_q) = 0$. 
\end{lemma}
\begin{proof}
If $F$ is an affine polynomial, that is, if $F = \lambda_{0,j} + \sum_{i=1}^q \lambda_{i,j} x_i$ for some $0 < j \leq \ell$, then because $\Dim{f_{4 + q + k + t + j}(a)(V)} = 0$, we have $\lambda_{0,j} + \sum_{i=1}^q \lambda_{i,j} h_i = 0$.    
On the other hand, if $F$ is a multiplicative relation, that is, if $F = x_{i_{1,j}} x_{i_{2,j}} - x_{i_{3,j}}$ for some $0< j \leq k$, then because $\Dim{f_{4+q+j}(a)(V)} = 1$ we see that there is some scalar $\gamma$ so that for any pair $(\alpha,\beta)$ we have $\gamma (\alpha h_{i_{1,j}} + \beta) = \alpha h_{i_{3,j}} + \beta h_{i_{2,j}}$.  
Specializing to $\beta = 0$ and $\alpha = 1$, we have $\gamma = \frac{h_{i_{3,j}}}{h_{i_{1,j}}}$ (unless both are zero, in which case the equality $h_{i_{1,j}} h_{i_{2,j}} = h_{i_{3,j}}$ holds anyway), which we substitute to obtain $\frac{h_{i_{3,j}}}{h_{i_{1,j}}} = h_{i_{2,j}}$, or $h_{i_{2,j}} h_{i_{1,j}} = h_{i_{3,j}}$. 
\end{proof}

Taking $b \ceq (h_{t+1},\ldots,h_{q})$ we see that $(a,b) \in V(S)(\QQ)$.  
\end{proof}

\section*{Acknowledgements}

We thank Bernd Sturmfels for suggesting
the connection with Hilbert's Tenth Problem over the rationals.

The research of various of the authors was supported by the U.S. Department of Energy Office of Science, 
Office of Advanced Scientific Computing Research, Applied Mathematics program under awards 
DE-SC0004938, DE-SC0003959, and DE-SC0010200; by the U.S. Department of Energy Office 
of Science, Office of Advanced Scientific Computing Research, X-Stack program under awards 
DE-SC0005136, DE-SC0008699, DE-SC0008700, and AC02-05CH11231; by 
DARPA award HR0011-12-2-0016 with contributions from Intel, Oracle, and MathWorks,
and by
NSF grants
DMS-1363324 and 
DMS-1001550.   

\bibliographystyle{plain}
{\small%
\bibliography{CDKSY15}}

\begin{thebibliography}{10}

\bibitem{bennettsharpley}
C.~Bennett and R.C. Sharpley.
\newblock {\em Interpolation of Operators}, volume 129 of {\em Pure and Applied
  Mathematics}.
\newblock Academic Press, 1988.

\bibitem{BCCT08}
J.~Bennett, A.~Carbery, M.~Christ, and T.~Tao.
\newblock {T}he {B}rascamp-{L}ieb inequalities: finiteness, structure and
  extremals.
\newblock {\em Geom. Funct. Anal.}, 17(5):1343--1415, 2008.

\bibitem{BCCT10}
J.~Bennett, A.~Carbery, M.~Christ, and T.~Tao.
\newblock Finite bounds for {H}{\"o}lder-{B}rascamp-{L}ieb multilinear
  inequalities.
\newblock {\em Mathematical Research Letters}, 17(4):647--666, 2010.

\bibitem{CLL04}
E.~A. Carlen, E.~H. Lieb, and M.~Loss.
\newblock {A} sharp analog of {Y}oung's inequality on ${S}^{N}$ and related
  entropy inequalities.
\newblock {\em Jour. Geom. Anal.}, 14:487--520, 2004.

\bibitem{christfiniteG2013}
M.~Christ.
\newblock Optimal constants in {H}{\"o}lder-{B}rascamp-{L}ieb inequalities for
  discrete {A}belian groups.
\newblock 2013.
\newblock arXiv preprint 1307.8442.

\bibitem{CDKSY14B}
M.~Christ, J.~Demmel, N.~Knight, T.~Scanlon, and K.~Yelick.
\newblock Communication lower bounds and optimal algorithms for programs that
  reference arrays.
\newblock {\em in preparation}, 2014.

\bibitem{Kon}
J.~Koenigsmann.
\newblock Defining $\mathbb{Z}$ in $\mathbb{Q}$.
\newblock {\em Annals of Mathematics}, to appear.
\newblock arXiv preprint 1011.3424.

\bibitem{Matiyasevich93}
Y.~Matiyasevich.
\newblock {\em Hilbert's Tenth Problem}.
\newblock MIT Press, 1993.

\bibitem{Vak}
R.~Vakil.
\newblock Murphy's law in algebraic geometry: badly-behaved deformation spaces.
\newblock {\em Inventiones Mathematicae}, 164(3):569--590, 2006.

\bibitem{Valdimarsson07}
S.I. Valdimarsson.
\newblock The {B}rascamp-{L}ieb polyhedron.
\newblock {\em Canadian Journal of Mathematics}, 62(4):870--888, 2010.

\end{thebibliography}
\end{document}